\documentclass[12pt]{article}
\usepackage[utf8]{inputenc}

\usepackage{amsmath,amssymb,amsthm} 
\usepackage{hyperref} 
\usepackage{xcolor} 
\usepackage{tikz} 
\usepackage{mathdots} 
\usepackage{tikz-cd} 

\usepackage{graphicx}

\newdimen\R
\newcommand{\Z}{\mathbb{Z}}

\newcommand{\Q}{\mathbb{Q}}

\newcommand{\Orb}{\mathcal{O}}

\newcommand{\fpq}{\frac{p}{q}}

\newtheorem{definition}{Definition}
\newtheorem{lemma}{Lemma}
\newtheorem{theorem}{Theorem}

\newtheorem{cor}{Corollary}

\newtheorem{prop}{Proposition}
\theoremstyle{remark}
\newtheorem{example}{Example}
\newtheorem{remark}{Remark}

\title{A Generalization of Markov Numbers}
\author{Esther Banaian and Archan Sen}
\date{}

\begin{document}

\maketitle

\abstract{We explore a generalization of the Markov numbers that is motivated by a specific generalized cluster algebra arising from an orbifold, in the sense of Chekhov and Shapiro. We give an explicit algorithm for computing these generalized Markov numbers and exhibit several patterns analogous to those that appear within the ordinary Markov numbers. Along the way, we present formulas related to continued fractions and snake graphs. }

\section{Introduction}

This article concerns a generalization of Markov numbers. Markov numbers appear in tuples which are solutions to a certain Diophantine equation. 

\begin{definition}\label{def:Markov}
A \emph{Markov tuple} is a tuple, $(a,b,c)$, of positive integers which satisfy $a^2 + b^2 + c^2 = 3abc$.  A number which appears in at least one Markov triple is called a \emph{Markov number}.  We call $x^2 + y^2 + z^2 = 3xyz$ the \emph{Markov equation.}
\end{definition}

Markov tuples and numbers first appeared in Markov's theorem in \cite{markoff1879formes}. They have remained of interest to mathematicians ever since, in no small part due to Frobenius' famous Uniqueness Conjecture which remains open to this day \cite{frobenius1913markoffschen}. This conjecture states that each Markov number is the largest number in a \emph{unique} Markov triple. For a history of work on this conjecture, see \cite{aigner1999proofs}.

The tuple $(1,1,1)$ is a natural first example of a Markov tuple. We can construct more Markov tuples by the following observation. Given a Markov tuple $(a,b,c)$, we can replace $c$ with $\frac{a^2 + b^2}{c}$ to obtain a distinct Markov tuple, $(a,b, \frac{a^2 + b^2}{c})$. This process is an example of \emph{Vieta jumping}. It is clear that $\frac{a^2 + b^2}{c}$ is positive, and since $3ab - c^2 = \frac{a^2 + b^2}{c}$, we also know this new number is an integer. Note that this process is an involution, and that we could similarly replace $a$ or $b$ with this method.
Furthermore, one can show that every Markov tuple is the result of applying a sequence of Vieta jumping to the Markov tuple $(1,1,1)$. 
A complete proof can be found in \cite{baragar1994integral}.

The process of going from a Markov tuple $(a,b,c)$ to another of the form $(a,b,\frac{a^2 + b^2}{c})$ is reminiscent of mutation in a cluster algebra. In \cite{beineke2011cluster} and \cite{propp2005combinatorics}, the respective authors explain the connection between Markov numbers and the cluster algebra arising from a once-punctured torus. Every triangulation of a once-punctured torus has three arcs and has adjacency quiver $Q_T$ (known as the Markov quiver), given below. 

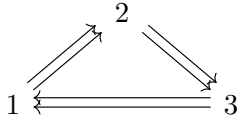
\begin{center}
\begin{tikzcd}
 & 2 \arrow[dr, shift left=.75ex] \arrow[dr] & \\
1 \arrow[ur, shift left=.75ex] \arrow[ur] && 3 \arrow[ll, yshift=.75ex] \arrow[ll]\\
\end{tikzcd}
\end{center}

More precisely, Markov numbers exactly correspond to the cluster variables in the cluster algebra from the Markov quiver when we set all initial cluster variables to 1. Since this cluster algebra arises from a surface, we can also interpret Markov numbers as the number of perfect matchings of \emph{snake graphs}, as were defined in \cite{MSW}.  Recall that a perfect matching of a graph $G = (V,E)$ is a subset of the edge set $P \subseteq E$ such that every vertex is incident to exactly one edge in $P$.  The snake graphs which correspond to Markov numbers were studied in detail in \cite{CS}.

In this paper, we study solutions to a Diophantine equation inspired by the Markov equation and its connection to the theory of cluster algebras. We refer to our equation as the \emph{generalized Markov equation} although there are many other interesting ways to generalize the equation. 

\begin{definition}\label{def:GenMarkov}
The \emph{generalized Markov equation} is \[
x^2 + y^2 + z^2 + xy + xz + yz = 6xyz.
\] 
A \emph{generalized Markov tuple} is a tuple of positive integers $(a,b,c)$ satisfying the generalized Markov equation. If $m$ is an element of at least one generalized Markov tuple, we call $m$ a \emph{generalized Markov number}.
\end{definition}

Generalized Markov tuples and numbers were studied by Gyoda in \cite{Gyoda} and in a broader context by Gyoda and Matsushita in \cite{GM}.
In particular, it is shown in \cite{Gyoda} that all generalized Markov tuples can be reached from the tuple of $(1,1,1)$ by exchanges similar to those used for the ordinary Markov equation.

\begin{theorem}[\cite{Gyoda}, Theorem 1.1]\label{thm:Gyoda}
Every generalized Markov tuple can be reached from the tuple $(1,1,1)$ by a sequence of exchanges of the form $(a,b,c) \to (a,b,\frac{a^2 + ab + b^2}{c})$.
\end{theorem}

The form of these exchanges resembles mutation in a \emph{generalized cluster algebra}, in the sense of Chekhov and Shapiro \cite{Chekhov}. The relevant generalized cluster algebra, $\mathcal{A}_3$, arises from a once-punctured sphere with three orbifold points, which we will denote $\mathcal{O}_3$. In parallel to the case of ordinary Markov numbers, generalized Markov numbers are given by generalized cluster variables in $\mathcal{A}_3$ when we specialize the initial cluster variables to 1. By following the construction of snake graphs from orbifolds in \cite{BK}, we can again interpret these generalized Markov numbers as the number of perfect matchings of snake graphs. This is the perspective we will take.

In Section \ref{sec:background}, we briefly give some of the background needed to explore our results. This includes discussion of snake graphs, continued fractions, and the labeling of ordinary Markov numbers via rational numbers $q$ such that $0 \leq q \leq 1$.
We will not define cluster algebras as the main results can be given without reference to cluster algebras. We direct a reader instead to the original papers on ordinary  \cite{fz} and generalized cluster algebras \cite{Chekhov} as well as a survey on ordinary cluster algebras by Glick and Rupel \cite{survey}.

 Our first main result is an algorithm to compute the number of perfect matchings of these snake graphs via continued fractions. The algorithm is outlined in Section \ref{sec:algorithm}, with the proof that it gives the correct continued fraction given in Theorem \ref{thm:OrdMarkovContFrac}. We use properties of these continued fractions to give both recurrences and growth behavior for certain sequences within the set of all generalized Markov numbers in Section \ref{sec:Patterns}.  The generalized Markov numbers correspond to arcs (with no self-intersection) on the orbifold associated to $\mathcal{A}$; in Section \ref{sec:extendalgorithm}, we consider generalized arcs (i.e. those with possibly self-intersection) and closed curves on this orbifold. Our final result in this section is Theorem \ref{thm:GenMarkovContFrac}, which gives recurrences on this extended family of numbers. In order to compute this recurrence, we provide Proposition \ref{prop:BandMatchings}, which gives a formula for computing the number of perfect matchings of a \emph{band graph}, and Theorem \ref{thm:BandMarkov}, which computes the number of good matchings of certain band graphs coming from the orbifold $\mathcal{O}_3$.

\section{Background}\label{sec:background}

\subsection{Labeling Markov Numbers with Rational Numbers}\label{subsec:Label}

There is a convenient labeling of Markov numbers larger than 1 using rational numbers in the interval $(0,1]$. One way to illustrate this labeling is by viewing triangulations and arcs on the universal cover of the once-punctured torus, $\mathbb{Z}^2$.  If our initial triangulation is $T_0 = \{\tau_1,\tau_2,\tau_3\}$,  with $\tau_2$ following $\tau_1$ in clockwise order, then we can lift $\tau_1$ to all line segments of the form $y = -x + a$, $\tau_2$ to all line segments of the form $x = b$, and $\tau_3$ to all line segments of the form $y = c$ for $a,b,c \in \mathbb{Z}$, where we consider segments between two consecutive lattice points. We will say these lines, and the arcs in the once-punctured torus which they represent, have slopes $\frac{-1}{1}, \frac10, \frac{0}{1}$ respectively, and we descriptively rename them  $\tau_{\frac{-1}{1}}, \tau_{\frac{1}{0}}, \tau_{\frac{0}{1}}$.

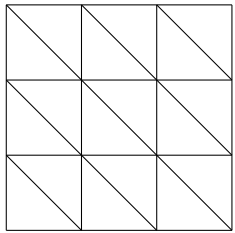
\begin{center}
\begin{tikzpicture}
\draw (0,0) -- (0,3);
\draw (1,0) -- (1,3);
\draw (2,0) -- (2,3);
\draw (3,0) -- (3,3);
\draw(0,0) -- (3,0);
\draw(0,1) -- (3,1);
\draw(0,2) -- (3,2);
\draw(0,3) -- (3,3);
\draw (3,0) -- (0,3);
\draw(2,0) -- (0,2);
\draw(1,0) -- (0,1);
\draw (3,1) -- (1,3);
\draw (3,2) -- (2,3);
\end{tikzpicture}
\end{center}

 We associate our initial triangulation $\{\tau_{\frac{-1}{1}},\tau_{\frac{1}{0}},\tau_{\frac{0}{1}}\}$ to the Markov tuple $(1,1,1)$. We can apply Vieta jumping to any number in this Markov tuple to reach $(1,1,2)$. For the arcs, we pick the convention that we flip $\tau_{\frac{-1}{1}}$. The flip of $\tau_{\frac{-1}{1}}$ will have slope $+1$ in the cover. Thus, we have that the Markov number labeled by $\frac{1}{1}$, $n_{\frac{1}{1}}$, is 2.

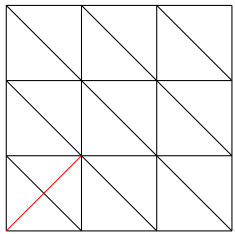
\begin{center}
\begin{tikzpicture}
\draw (0,0) -- (0,3);
\draw (1,0) -- (1,3);
\draw (2,0) -- (2,3);
\draw (3,0) -- (3,3);
\draw(0,0) -- (3,0);
\draw(0,1) -- (3,1);
\draw(0,2) -- (3,2);
\draw(0,3) -- (3,3);
\draw (3,0) -- (0,3);
\draw(2,0) -- (0,2);
\draw(1,0) -- (0,1);
\draw (3,1) -- (1,3);
\draw (3,2) -- (2,3);
\draw[red] (0,0) to (1,1);
\end{tikzpicture}
\end{center}

Since Vieta jumping is an involution, at the Markov tuple $(1,1,2)$ we must apply Vieta jumping to one of the entries of $1$,  reaching $(1,2,5)$. In the triangulation $\{\tau_{\frac{1}{1}},\tau_{\frac{1}{0}},\tau_{\frac{0}{1}}\}$, we pick the convention that we will flip $\tau_{\frac{1}{0}}$, which was associated with lines of slope $\frac{1}{0}$. The resulting new arc will have slope $\frac12$. From this we have $n_{\frac{1}{2}} = 5$. 

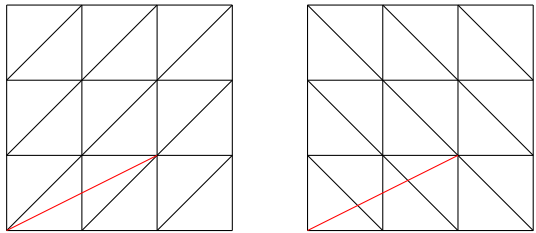
\begin{center}
\begin{tikzpicture}
\draw (0,0) -- (0,3);
\draw (1,0) -- (1,3);
\draw (2,0) -- (2,3);
\draw (3,0) -- (3,3);
\draw(0,0) -- (3,0);
\draw(0,1) -- (3,1);
\draw(0,2) -- (3,2);
\draw(0,3) -- (3,3);
\draw(0,0) -- (3,3);
\draw (3,2) -- (1,0);
\draw(3,1) -- (2,0);
\draw(2,3) -- (0,1);
\draw (1,3) -- (0,2);
\draw[red] (0,0) to (2,1);
\draw[xshift = 5\R] (0,0) -- (0,3);
\draw[xshift = 5\R] (1,0) -- (1,3);
\draw[xshift = 5\R] (2,0) -- (2,3);
\draw[xshift = 5\R] (3,0) -- (3,3);
\draw[xshift = 5\R](0,0) -- (3,0);
\draw[xshift = 5\R](0,1) -- (3,1);
\draw[xshift = 5\R](0,2) -- (3,2);
\draw[xshift = 5\R](0,3) -- (3,3);
\draw[xshift = 5\R] (3,0) -- (0,3);
\draw[xshift = 5\R](2,0) -- (0,2);
\draw[xshift = 5\R](1,0) -- (0,1);
\draw[xshift = 5\R] (3,1) -- (1,3);
\draw[xshift = 5\R] (3,2) -- (2,3);
\draw[red, xshift = 5\R] (0,0) to (2,1);
\end{tikzpicture}
\end{center}

At this point, if we continue to flip arcs in the torus and do not flip the same arc two times in a row, their lifts will always have slope less than 1. This is a consequence of the following lemma.

\begin{lemma}\label{lem:SlopeArcs}
\begin{enumerate}
\item The set of slopes of each triangulation of the once-punctured torus are of the form $\{ \frac{a}{c}, \frac{a+b}{c+d}, \frac{b}{d}\}$.
\item Mutation has the following effect on the slopes of a triangulation \[\{ \frac{a}{c}, \frac{a+b}{c+d}, \frac{b}{d}\} \to \{ \frac{(a+b) + b}{(c+d) + d}, \frac{a+b}{c+d}, \frac{b}{d}\}\]
\end{enumerate}
\end{lemma}

The operation of combining $\frac{a}{c}$ and $\frac{b}{d}$ to $\frac{a+b}{c+d}$ is referred to as a \emph{Farey sum}. Triangles of the form mentioned in Lemma \ref{lem:SlopeArcs} form the \emph{Farey tesselation} of the upper-half plane. 

We display the exchange trees for Markov numbers and rational numbers with respect to Vieta jumping and the Farey sum respectively in Figure \ref{fig:MarkovAndQTrees}.  For example, we can see that $n_{\frac23} = 29$ by noting the positions where $\frac23$ and $29$ first appear in each tree. {\c{C}}anak{\c{c}}{\i} and Schiffler discuss a combinatorial way to compute the Markov number associated to each rational number in \cite{CS}.  

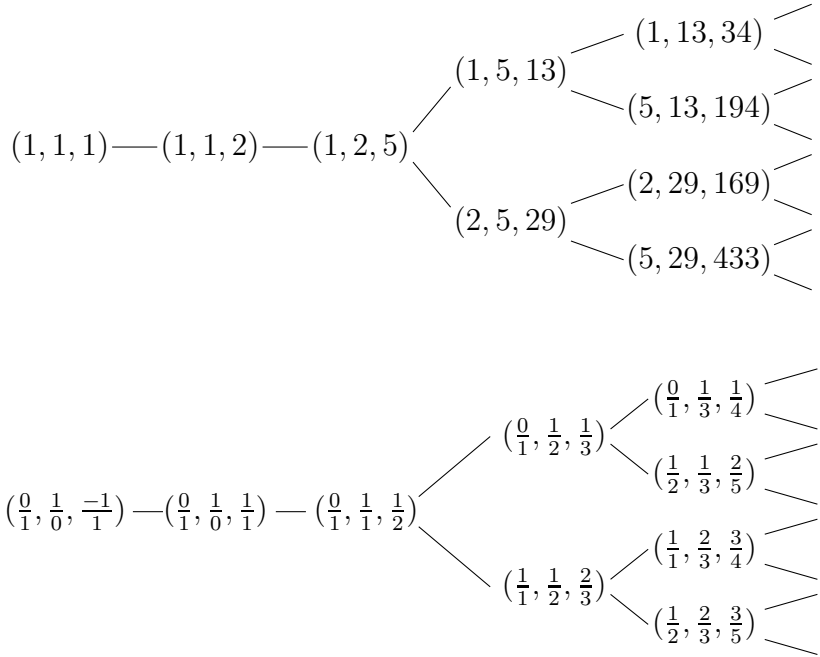
\begin{figure}
\centering
\begin{tikzpicture}
\node[] at (-1,0){$(1,1,1)$};
\node[] at (1,0){$(1,1,2)$};
\node[] at (3,0){$(1,2,5)$};
\node[] at (5,1){$(1,5,13)$};
\node[] at (5,-1){$(2,5,29)$};
\node[] at (7.5,1.5){$(1,13,34)$};
\node[] at (7.5,0.5){$(5,13,194)$};
\node[] at (7.5,-0.5){$(2,29,169)$};
\node[] at (7.5,-1.5){$(5,29,433)$};
\draw (0.7-1,0) -- (1.3-1,0);
\draw(2.7-1,0) -- (3.3-1,0);
\draw (4.2-.5,0.2) -- (5.2-1,0.8);
\draw(4.2-.5,-0.2) -- (5.2-1,-0.8);
\draw(5.8,1.25) -- (6.5,1.5);
\draw(5.8,0.75) -- (6.5,0.5);
\draw(5.8,-0.75) -- (6.5,-0.5);
\draw(5.8,-1.25) -- (6.5, -1.5);
\draw(8.5,1.7) -- (9,1.9);
\draw(8.5,1.3) -- (9,1.1);
\draw (8.5,0.7) -- (9,0.9);
\draw(8.5,0.3) -- (9,0.1);
\draw(8.5,-0.3) -- (9,-0.1);
\draw(8.5,-0.7) -- (9,-0.9);
\draw(8.5,-1.3) -- (9,-1.1);
\draw(8.5,-1.7) -- (9,-1.9);
\end{tikzpicture}

\vspace{1cm}

\begin{tikzpicture}
\node[] at (0-1,0){$(\frac01,  \frac10, \frac{-1}{1})$};
\node[] at (2-1,0){$(\frac01,  \frac10, \frac{1}{1})$};
\node[] at (4-1,0){$(\frac01,  \frac11, \frac{1}{2})$};
\node[] at (5.5,1){$(\frac01, \frac12, \frac13)$};
\node[] at (5.5,-1){$(\frac11, \frac12, \frac23)$};
\node[] at (7.5,1.5){$(\frac01,\frac13,\frac14)$};
\node[] at (7.5,0.5){$(\frac12, \frac13,\frac25)$};
\node[] at (7.5,-0.5){$(\frac11, \frac23, \frac34)$};
\node[] at (7.5,-1.5){$(\frac12, \frac23, \frac35)$};
\draw (0.9-1,0) -- (1.3-1,0);
\draw(2.8-1,0) -- (3.2-1,0);
\draw (4.2-.5,0.2) -- (4.65,1);
\draw(4.2-.5,-0.2) -- (4.65,-1);
\draw(6.25,1.1) -- (6.75,1.5);
\draw(6.25,0.9) -- (6.75,0.5);
\draw(6.25,-0.9) -- (6.75,-0.5);
\draw(6.25,-1.1) -- (6.75, -1.5);
\draw(8.3,1.7) -- (9,1.9);
\draw(8.3,1.3) -- (9,1.1);
\draw (8.3,0.7) -- (9,0.9);
\draw(8.3,0.3) -- (9,0.1);
\draw(8.3,-0.3) -- (9,-0.1);
\draw(8.3,-0.7) -- (9,-0.9);
\draw(8.3,-1.3) -- (9,-1.1);
\draw(8.3,-1.7) -- (9,-1.9);
\end{tikzpicture}
\caption{The initial portions of the exchange trees for Markov Numbers and $\Q \cap (0,1)$}
\label{fig:MarkovAndQTrees}
\end{figure}

Since all generalized Markov numbers are also reachable by a sequence of Vieta jumping, we can also index generalized Markov numbers  larger than 1 with rational numbers in the interval $(0,1]$.  By comparing Figures \ref{fig:MarkovAndQTrees} and \ref{fig:GenMarkovTree}, we can for instance see that the generalized Markov number associated to $\frac23$, $m_{\frac23}$, is 217. In Section \ref{sec:algorithm}, we will give a direct way to compute $m_{\frac{p}{q}}$. 

\begin{figure}
\centering
\begin{tikzpicture}
\node[] at (0,0){$(1,1,1)$};
\node[] at (2,0){$(1,1,3)$};
\node[] at (4.2,0){$(1,3,13)$};
\node[] at (6.6,1){$(1,13,61)$};
\node[] at (6.6,-1){$(3,13,217)$};
\node[] at (9.5,1.5){$(1,61,291)$};
\node[] at (9.5,0.5){$(13,61,4683)$};
\node[] at (9.5,-0.5){$(3,217,3673)$};
\node[] at (9.5,-1.5){$(13,217,16693)$};
\draw (0.7,0) -- (1.3,0);
\draw(2.7,0) -- (3.4,0);
\draw (5,0.2) -- (5.6,0.9);
\draw(5,-0.2) -- (5.6,-0.9);
\draw(7.5,1.2) -- (8.4,1.5);
\draw(7.5,0.8) -- (8.3,0.5);
\draw(7.6,-0.8) -- (8.3,-0.5);
\draw(7.6,-1.2) -- (8, -1.5);
\draw(10.5,1.7) -- (10.8,1.9);
\draw(10.5,1.4) -- (10.8,1.2);
\draw (10.7,0.7) -- (11,0.9);
\draw(10.7,0.4) -- (11,0.2);
\draw(10.7,-0.4) -- (11,-0.2);
\draw(10.7,-0.7) -- (11,-0.9);
\draw(10.9,-1.4) -- (11.2,-1.2);
\draw(10.9,-1.7) -- (11.2,-1.9);
\end{tikzpicture}
\caption{The initial portions of the exchange tree for generalized Markov tuples}
\label{fig:GenMarkovTree}
\end{figure}
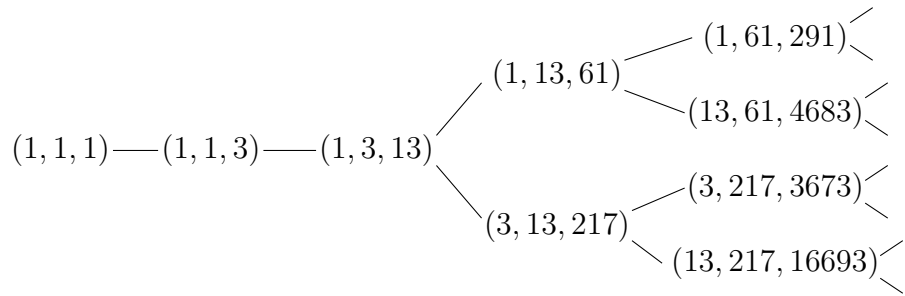

\subsection{Snake Graphs on Surfaces and Orbifolds}

Consider a connected,  oriented 2-dimensional Riemann surface  $S$ with a finite subset $M$ of points called \emph{marked points}, and pick a triangulation $T$ of the surface. As is thoroughly studied in \cite{FST} and \cite{FT}, we can associate a cluster algebra to the tuple $(S,M)$, such that the arcs are in bijection with cluster variables; in particular, arcs in $T$ are associated to initial cluster variables. In \cite{MSW}, Musiker, Schiffler, and Williams gave a direct way to compute the cluster variable $x_\gamma$ associated to the arc $\gamma$ via an edge-labeled graph called a \emph{snake graph}. If $\gamma$ has crosses arcs $\tau_{i_1},\ldots, \tau_{i_d}$ in the triangulation $T$, then the snake graph $G_{\gamma,T}$ consists of $d$ square tiles, $G_1,\ldots,G_d$,  glued along edges. The  tile $G_j$ represents the quadrilateral around the arc $\tau_{i_j}$ in $T$. Then, we can calculate $x_\gamma$ with respect to the initial cluster associated to $T$ by looking at all perfect matchings of $G_{\gamma,T}$. In the following, $\text{cross}(\gamma,T)$, $x(P)$ is the product of the weights of the edges in $P$, and $y(P)$ is another statistic associated to a perfect matching. We do not dwell on these details as they will not be necessary for our main results. 

\begin{theorem}[\cite{MSW}, Theorem 4.9]\label{thm:SnakeGraphSurface}
Given a triangulation $T$ on a surface with marked points $(S,M)$, let $\gamma$ be an arc on $(S,M)$. Then, the expansion of the cluster variable $x_\gamma$ in the cluster algebra arising from $(S,M)$ with initial cluster from $T$ is given by \[
x_\gamma^T = \frac{1}{\text{cross}(\gamma,T)} \sum_{P} x(P)y(P)
\]
where we sum over perfect matchings $P$ of $G_{\gamma,T}$. 
\end{theorem}

For an example of a snake graph,  see Figure \ref{fig:G1q}. 

\subsection{Orbifolds}

An orbifold is a generalization of a manifold where the local structure is given by quotients of open subsets of $\mathbb{R}^n$ under finite group actions. For our considerations, an orbifold $(S,M,Q)$ is a marked surface with an additional set of special points called \emph{orbifold points} $Q$. Each orbifold point comes with an \emph{order} $p \in \mathbb{Z}_{\geq 2}$. Arcs in an orbifold have endpoints in $M$ and cannot pass through orbifold points. An arc which cuts out an unpunctured monogon with exactly one point in $Q$ is called a \emph{pending arc}.

Part of the story about cluster algebras from surfaces was extended to orbifolds by Chekhov and Shapiro in \cite{Chekhov}.  The relevant algebra is a \emph{generalized cluster algebra}.  These cluster algebras are generalized in the sense that the exchange polynomials can have more than 2 terms. A generalized cluster algebra arising from an orbifold will have exchange polynomials with 2 or 3 terms; the number of terms depends on whether the variable corresponds to a standard arc or a pending arc. In particular, in the generalized cluster algebra arising from $\mathcal{O}_3$, the exchange polynomials are all of the form $u+v$ or $u^2 + uv + v^2$. The snake graph expansion formula was extended to generalized cluster variables in a generalized cluster algebra from an orbifold by the first author and Elizabeth Kelley \cite{BK}. 

\begin{theorem}[\cite{BK}, Theorem 1.1]\label{thm:BK}
Given a triangulation $T$ on an orbifold $\mathcal{O}$, let $\gamma$ be an arc on $\mathcal{O}$. Then, the expansion of the generalized cluster variable $x_\gamma$ in the cluster algebra arising from $\mathcal{O}$ with initial cluster from $T$ is given by \[
x_\gamma^T = \frac{1}{\text{cross}(\gamma,T)} \sum_{P} x(P)y(P)
\]
where we sum over perfect matchings $P$ of $G_{\gamma,T}$. 
\end{theorem}

In this paper we specifically consider the orbifold $\mathcal{O}_3$, a sphere with one puncture and three orbifold points of order three.  Every triangulation of $\mathcal{O}_3$ consists of three pending arcs which are all based at the unique marked point. Arcs on an orbifold can be flipped just as arcs on a surface; moreover, a flip of a pending arc will always be another pending arc. Thus we see that the only effect that a flip has on the initial triangulation on $\mathcal{O}_3$ is flipping the relative orientation of the arcs in the triangulation, just as in the case of a once-punctured torus. 

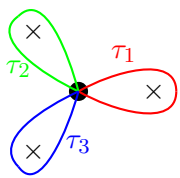
\begin{center}
    \begin{tikzpicture}
    \node[circle, fill = black, scale = 0.7] at (0,0){};
    \node[] at (1,0){$\times$};
    \node[] at (-0.6,0.8){$\times$};
    \node[] at (-0.6, -0.8){$\times$};
    \draw[red, thick] (0,0) to [out = -30, in = 270] (1.3,0);
    \draw[red, thick] (0,0) to [out = 30, in = 90] (1.3,0);
    \node[red] at (0.6,0.5){$\tau_1$};
    \draw[green,thick] (0,0) to [out = 105, in = 45] (-0.8,1);
    \draw[green,thick] (0,0) to [out = 155, in = 215]  (-0.8,1);
    \node[green] at (-0.8,0.3){$\tau_2$};
    \draw[blue,thick] (0,0) to [out = 255, in = 305] (-0.8,-1);
    \draw[blue,thick] (0,0) to [out = 205, in = 135] (-0.8,-1);
    \node[blue] at (0,-0.7){$\tau_3$};
    \end{tikzpicture}
\end{center}

The combinatorics of the orbifold $\mathcal{O}_3$, Theorem \ref{thm:Gyoda} and Theorem \ref{thm:BK} allow us to interpret the generalized Markov numbers and tuples in terms of a generalized cluster algebra.  Given a graph $G$,  let $m(G)$ be the number of perfect matchings of $G$. 

\begin{cor}[\cite{Gyoda}, Corollary 3.9]\label{cor:Gyoda}
Let $\gamma_1,\gamma_2,\gamma_3$ be a triangulation of $\mathcal{O}_3$.
\begin{itemize}
    \item For any $i$, the number $m(G_{\gamma_i,T})$ is a generalized Markov number. All generalized Markov numbers appear in this way.
    \item  The triple $(m(G_{\gamma_1,T}),m(G_{\gamma_2,T}),m(G_{\gamma_3,T}))$ is a generalized Markov triple. All generalized Markov tuples appear in this way.
\end{itemize}
\end{cor}

\begin{remark}
The description of generalized snake graphs in \cite{BK} was only concerned with orbifolds without punctures. When punctures are present in a surface, in order to consider the corresponding cluster algebra one must use \emph{tagged arcs} and \emph{tagged triangulations}. However, any triangulation of $\Orb_3$ consists of three pending arcs, which must all have the same tagging at the unique puncture.  Thus, if we consider an initial triangulation with all arcs tagged plain, we can use generalized snake graphs to compute the generalized cluster variables resulting from any finite sequence of mutations. 
\end{remark}

\subsection{Continued Fractions}

{\c{C}}anak{\c{c}}{\i} and Schiffler give a method to compute the number of perfect matchings of a snake graph via continued fractions in \cite{CS}. This method first requires knowledge of a \emph{sign function} on a snake graph. 

\begin{definition}
Given a snake graph $G$,  a \emph{sign function} on $G$ labels all edges of the snake graph with $+$ or $-$ such that \begin{enumerate}
    \item the signs on the South and East edges of a tile are the same,
    \item the signs on the North and West edges of a tile are the same, 
    \item the signs on the North and South edges of a tile are the opposite, and
    \item the signs on the East and West edges of a tile are the opposite. 
\end{enumerate}
In order to have a unique sign function on each snake graph, we instill the convention that the South edge of the first tile is assigned the sign $-$. 

The \emph{sign sequence} of a snake graph $G$ with $m-1$ tiles is $f_1,\ldots,f_m$ where $f_1$ is the sign on the south edge of the first tile ($-$ by convention), for $1 < i < m$, $f_i$ is the sign of the edge shared by tiles $G_i$ and $G_{i+1}$, and $f_m = f_{m-1}$. 
\end{definition}

We call the edge shared by tiles $G_i$ and $G_{i+1}$ an \emph{internal edge}.

A consequence of the definition of a sign function is that the sign is constant on a diagonal line through the edges of the snake graph traveling North-East. Moreover, we can see that a snake graph is determined by its sign sequence.  

From the sign sequence, we can compute the number of perfect matchings of a snake graph by looking at the numerator of an associated \emph{continued fraction}.

\begin{definition}
Let $a_1,\ldots,a_n$ be a sequence of integers. Then, the \emph{continued fraction}, $[a_1,\ldots,a_n]$ is given by \[
[a_1,\ldots,a_n] = a_1 + \cfrac{1}{a_2 + \cfrac{1}{a_3 + \cfrac{1}{\ddots + \cfrac{1}{a_n}}}}
\]
\end{definition}

Every rational number can be expressed as a continued fraction. This expression is unique if we require $a_n > 1$; it is straightforward to see that if $a_n > 1$, $[a_1,\ldots,a_n] = [a_1,\ldots,a_n - 1,1]$. 

Given a snake graph with sign sequence $f_1,\ldots,f_m$,  we form a continued fraction $[a_1,\ldots,a_n]$ by counting the lengths of subsequences of the sign sequence which have the same sign.  For example, if the sign sequence is $--+++--$,  then the continued fraction will be $[2,3,2]$. Since  $[a_1,\ldots,a_n] = [a_1,\ldots,a_n - 1,1]$,  we can choose $f_{m-1} = f_m$ in the sign sequence from a snake graph $G$ on $m$ tiles. Pictorially, we can choose the sign of either the North or East edge of the last tile, so we always choose the one which has the same sign as the internal edge of the last tile. 

Let $\mathcal{N}([a_1,\ldots,a_n])$ be the numerator of the continued fraction $[a_1,\ldots,a_n]$. Let $\mathcal{G}[a_1,\ldots,a_n]$ be the snake graph whose sign function gives the continued fraction $[a_1,\ldots,a_n]$; after fixing our conventions, this snake graph is unique.

\begin{theorem}[\cite{CS}, Theorem A]\label{thm:ContFracSnakeGraph}
Let $a_1,\ldots,a_n$ be a sequence of positive integers such that $a_n > 1$. Then, \[
[a_1,\ldots,a_n] = \frac{m(\mathcal{G}[a_1,\ldots,a_n])}{m(\mathcal{G}[a_2,\ldots,a_n])},
\]
where the right-hand side is reduced.
In particular, $m(\mathcal{G}[a_1,\ldots,a_n]) = \mathcal{N}[a_1,\ldots,a_n]$. 
\end{theorem}

\begin{example}

We display the entire sign function on the following 5-tile snake graph. The sign sequence from this snake graph is $-,-,-,+,+,+$. 

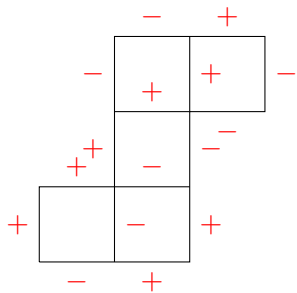
\begin{center}
    \begin{tikzpicture}
    \draw(0,0) to node[below, red]{$-$} (1,0) to node[below, red]{$+$} (2,0) to node[right, red]{$+$} (2,1) to node[right, red]{$-$} (2,2) to  node[below, red]{$-$} (3,2) to  node[right, red]{$-$} (3,3) to node[above, red]{$+$} (2,3) to  node[above, red]{$-$}(1,3) to  node[left, red]{$-$}(1,2) to  node[left, red]{$+$} (1,1) to node[above, red]{$+$} (0,1) to node[left, red]{$+$} (0,0);
    \draw (1,0) to node[right, red]{$-$} (1,1);
    \draw (2,1) to node[above, red]{$-$} (1,1);
    \draw (2,2) to node[above, red]{$+$} (1,2);
    \draw (2,3) to node[right, red]{$+$} (2,2);
    \end{tikzpicture}
\end{center}

The continued fraction to compute in this case is $[3,3]$. We have that $[3,3] = 3 + \frac13 = \frac{10}{3}$, and indeed this snake graph has 10 perfect matchings. 
\end{example}

We record a few results concerning continued fractions which will be useful for our later calculations. The first two are well-known.

\begin{lemma}\label{lem:FlipContFraction}
Let $a_1,\ldots,a_n$ be positive integers. Then,
\[
\mathcal{N}[a_1,\ldots,a_n] = \mathcal{N}[a_n,\ldots,a_1].
\]
\end{lemma}

\begin{lemma}\label{lem:ContFracRecurrence}
Let $n \geq 2$ and let $a_1,\ldots,a_n$ be positive integers. Then,
\[
\mathcal{N}[a_1,\ldots,a_n] = a_1 \mathcal{N}[a_2,\ldots,a_n] + \mathcal{N}[a_3,\ldots,a_n]
\]
and 
\[
\mathcal{N}[a_1,\ldots,a_n] = a_n \mathcal{N}[a_1,\ldots,a_{n-1}] + \mathcal{N}[a_1,\ldots,a_{n-2}].
\]
where we define $\mathcal{N}[] = 1$.  
\end{lemma}

It is clear that $[a_1,\ldots,a_n,1] = [a_1,\ldots,a_n+1]$.  Combining this with Lemma \ref{lem:FlipContFraction} gives us a result for continued fractions with 1 as the first entry.

\begin{lemma}\label{lem:OneAtEnd}
Let $a_1,\ldots,a_n$ be positive integers. Then,\[
\mathcal{N}[1,a_1,\ldots,a_n] = \mathcal{N}[a_1+1,\ldots,a_n].
\]
\end{lemma}

The final result in this section will be useful in the proof of Theorem \ref{thm:BandMarkov}.

\begin{lemma}\label{lem:SwapPlus1}
Let $k \geq 2$ and let $a_1,\ldots,a_k$ be positive integers. Then,
\[
\mathcal{N}[a_2,\ldots,a_k+1, a_k , \ldots, a_1] = \mathcal{N}[a_1,\ldots,a_k+1,a_k,\ldots,a_2] + c_k
\]
where $c_k$ is given by \[
c_k = \begin{cases} 1 & k \text{ is even}\\
-1 & k \text{ is odd}
\end{cases}
\]
\end{lemma}

\begin{proof}
We prove this by induction. For the base case, we directly compute \[
\mathcal{N}[a_1, a_2, a_2+1] = a_1(a_2^2 + a_2+1)+a_2 + 1
\]
and \[
\mathcal{N}[a_1,a_2+1,a_2] = a_1(a_2^2 + a_2+1)+a_2.
\]

Now, we assume the claim for $k-1$ and analyze the case for $k$. We expand the first term using Lemma \ref{lem:ContFracRecurrence} , \begin{align*}
\mathcal{N}[a_2,\ldots,a_k+1, a_k, \ldots, a_1] \\= a_1 \mathcal{N}[a_2,\ldots,a_k+1, a_k, \ldots, a_2] + \mathcal{N}[a_2, \ldots,a_k+1, a_k, \ldots, a_3]\\
=a_1 \mathcal{N}[a_2,\ldots,a_k+1, a_k, \ldots, a_2] + \mathcal{N}[a_3, \ldots,a_k+1, a_k , \ldots, a_2] - c_{k-1},
\end{align*}

where the last equality comes from applying the inductive hypothesis.

Then we also apply Lemma \ref{lem:ContFracRecurrence} to the continued fraction on the lefthand side.
\begin{align*}
 \mathcal{N}[a_1, a_2,\ldots,a_k+1, a_k, \ldots, a_2]  \\
 = a_1 \mathcal{N}[a_2,\ldots,a_k, a_k + 1, \ldots, a_2] + \mathcal{N}[a_3, \ldots,a_k+1, a_k, \ldots, a_2]\\
\end{align*}

Thus, we see that $\mathcal{N}[a_2,\ldots,a_k+1,a_k,\ldots,a_1] = \mathcal{N}[a_1,\ldots,a_k+1,a_k,\ldots,a_2] - c_{k-1}$.  Therefore,  $c_k = -c_{k-1}$; since the base case showed $c_2 = 1$,  we are done.
\end{proof}

\section{Algorithm}\label{sec:algorithm}

Recall from Section \ref{subsec:Label} that we index the generalized Markov numbers with (reduced) rational numbers $\frac{p}{q}$ in the interval $(0,1]$. However, given such a rational number $\frac{p}{q}$, it is not immediately clear what the generalized Markov number $m_{\frac{p}{q}}$ is. In this section, we give a direct way to compute $m_{\frac{p}{q}}$. This method is inspired by the shape of the snake graph $G_{\gamma_{\frac{p}{q}}, T}$ for the arc, $\gamma_{\frac{p}{q}}$, which is associated to $m_{\frac{p}{q}}$. Such an arc is guaranteed in Corollary \ref{cor:Gyoda}. However, it is more compact to give the continued fraction $[a_1,\ldots,a_n]$ associated to $G_{\gamma_{\frac{p}{q}}, T}$, so we use this language instead.  {\c{C}}anak{\c{c}}{\i} and Schiffler give a similar construction in \cite{CS}.

In what follows, we will describe a function $\mathbf{f}$ which maps from $\mathbb{Q} \cap (0,1]$ to sequences from the alphabet $\{-,+\}$ of varied length. In particular,  $\mathbf{f}(\frac{p}{q}) = (f_0,\ldots,f_{4(p+q)-6})$ where $f_i \in \{-,+\}$. In Theorem \ref{thm:GenMarkovContFrac}, we show that $\mathbf{f}$ is the sign sequence associated to $G_{\gamma_{\frac{p}{q}},T}$.

Throughout,  we fix the lattice given by all lines through integral points with slopes $0, 1,$ and $\infty$.  We will consider each line segment between consecutive pairs of integer points as a distinct arcs.  

Let $\gamma_{\frac{p}{q}}$ denote the line segment between $(0,0)$ and $(q,p)$, oriented to start at $(0,0)$ and end at $(q,p)$.  We can describe $\mathbf{f}(\frac{p}{q}) = (f_0,\ldots,f_{4(p+q)-6})$ as the following. We always set $f_0 = -$ and $f_{4(p+q)-6} = +$. For $1 \leq i \leq 2(p+q)-3$, let $\sigma_i$ be the $i$-th line segment in the lattice which $\gamma_{\frac{p}{q}}$ crosses, and let $s_i$ be this intersection point. The odd-indexed entries keep track of whether $s_i$ is closer to endpoint of $\sigma_i$ which lies to the right or to the left of $\gamma_{\frac{p}{q}}$; if it is closer to the endpoint on the right, we assign $f_{2i+1} = -$ and otherwise we assign $f_{2i+1} = +$. There is only one intersection point which is at the midpoint of an arc, and we will see we can choose either sign here. The even-indexed entries record whether the endpoint shared by $\sigma_i$ and $\sigma_{i+1}$ lies to the right or left of $\gamma_{\frac{p}{q}}$. Again, if this endpoint is to the right, then assign $f_{2i} = -$, and  assign $f_{2i} = +$ if it lies to the left. 

In the following, we give formulas to directly compute $\mathbf{f}(\frac{p}{q})$. First, we recall notation given in \cite{CS}. 
\begin{align*}
    v_1 = \lfloor \frac{q}{p} \rfloor \\
    v_i = \lfloor \frac{qi}{p} \rfloor - (v_1 + \cdots + v_{i-1}) \quad 1 < i < p\\
    v_p = (q-1) - (v_1 + \cdots + v_{p-1})
\end{align*}

The quantity $v_i$ tell us how many vertical lines $\gamma_{\frac{p}{q}}$ crosses between its crossing of the horizontal lines $y = i-1$ and $y = i$. For $0 < i < p$, we define \[
 \widetilde{v_i} = \lfloor \frac{qi}{p} \rfloor+ i
 \]
  The $\widetilde{v_i}$ give information about which arcs $\sigma_j$ are horizontal. In particular $\sigma_j$ is horizontal if and only if $j = 2\widetilde{v_i}$.  It will later be convenient to set $\widetilde{v_0} = 0$ and $\widetilde{v_p} = p+q-1$. 
 
As in \cite{CS}, we can compute the even-indexed terms $f_{2i}$ using the $v_i$ and $\widetilde{v_i}$. First we specify $f_{2i}$ if there exists $0 < j< p$ such that $2\widetilde{v_j} - 2 \leq i \leq 2\widetilde{v_j} + 1$,. This corresponds to the crossings near a crossing of $\gamma_{\frac{p}{q}}$ and a horizontal line segment; in particular, since the slope is less than 1, we know that the arcs crossed by $\gamma_{\frac{p}{q}}$ near a crossing with a horizontal arc will have the pattern: vertical, diagonal, horizontal, diagonal, vertical. Thus, we set  \begin{align*}
f_{2(2\widetilde{v_j} - 2)} = f_{2(2\widetilde{v_j} - 1)} = +\\
f_{4\widetilde{v_j}} = f_{2(2\widetilde{v_j}+1)} = -\\
\end{align*}

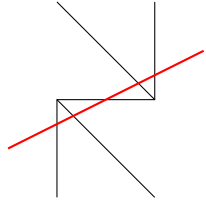
\begin{center}
\begin{tikzpicture}[scale = 1.3]
\draw (0,0) -- (0,1) -- (1,1) -- (1,2);
\draw (1,0) -- (0,1);
\draw (0,2) -- (1,1);
\draw[red, thick] (-0.5,0.5) -- (1.5,1.5);
\end{tikzpicture}
\end{center}

The rest of the even-indexed entries correspond to $\gamma_{\frac{p}{q}}$ crossing an alternating sequence of vertical and diagonal edges. Precisely, if for some $0 \leq j < p$, $i$ satisfies $2\widetilde{v_j} + 1 < i < 2\widetilde{v_{j+1}} - 2$, then \[
f_{2i} = \begin{cases} - & i \text{ is even} \\ + & \text{otherwise.} \end{cases}
\]

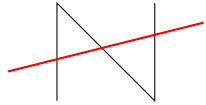
\begin{center}
\begin{tikzpicture}[scale = 1.3]
\draw (0,0) -- (0,1) -- (1,0) -- (1,1);
\draw[red, thick] (-0.5,0.3) -- (1.5,0.8);
\end{tikzpicture}
\end{center}

Finally, we also set $f_0 = f_2 = -$ and $f_{4(p+q)-6}  = +$. 

Next we turn to the odd-indexed entries. The entries $f_j$ for $j \equiv 3 \pmod{4}$ of $\mathbf{f}(\frac{p}{q})$ record information about the crossing points of $\gamma_{\frac{p}{q}}$ and horizontal and vertical line segments. We know that the segment $\sigma_{2\widetilde{v_j}}$ lies on the line $y = j$. Thus, we compute whether the intersection of $\gamma_{\frac{p}{q}}$ and $\sigma_{2\widetilde{v_j}}$ is closer to the right or left endpoint with the following, 
\[
f_{4(\widetilde{v_j} - 1) + 3} = \begin{cases}
- &\lceil \frac{jq}{p} \rceil-  \frac{jq}{p}  < \frac12 \\
+ & \text{ otherwise.} \\
\end{cases}
\]

All other entries $f_i$ with $i \equiv 3 \pmod{4}$ record information about the crossing of $\gamma_{\frac{p}{q}}$ and a vertical edge. If $i \neq \widetilde{v_j}$ for any $1 \leq j \leq p-1$, let $i'$ be the largest integer in $[0,p-1]$ such that $i > \widetilde{v_{i'}} $. This implies that arc $\sigma_{2j}$ lies on the line $x = i - i'$. Therefore, we set \[
f_{4(i - 1) + 3} = \begin{cases}
- & \ \frac{(i-i')p}{q}  - \lfloor\frac{(i-i')p}{q}\rfloor< \frac12 \\
+ & \text{ otherwise.} \\
\end{cases}
\]

Finally, we look at the intersections of $\gamma_{\frac{p}{q}}$ and diagonal arcs. Note that all arcs $\sigma_j$ for odd $j$ are diagonal. We define $w_i$ for $1 \leq i \leq p+q-1$,
 
 \[
 w_i = \frac{qi}{p+q}.
 \]
 
 Since $\gamma_{\frac{p}{q}}$ and $y = -x + i$ intersect at $(w_i, \frac{p}{q} w_i)$, we set 
 \[
 f_{4(i-1)+1} = \begin{cases} - & \lceil w_i \rceil - w_i < \frac12 \\
 + & \text{ otherwise.}\\
 \end{cases}
 \]
 
In Table \ref{table:contfrac}, we give the continued fractions for the sign sequences $\mathbf{f}(\frac{p}{q})$ and we give the numerators of these continued fractions in Table \ref{table:GenMarkov}.  The values associated to $(p,q)$ with $\gcd(p,q) > 1$ in both tables will be explained in Section \ref{sec:extendalgorithm}.

\begin{table}
\renewcommand{\arraystretch}{1.5}
\centering
\begin{tabular}{|c||c|c|c|c|c|c|c|}
     \hline
     $p \backslash q$ & 1&2&3&4&5&6&7 \\\hline\hline
     1 & \scalebox{1}{$3$} & \scalebox{1}{$13$} & \scalebox{1}{$61$} & \scalebox{1}{$291$} & \scalebox{1}{$1393$} & \scalebox{1}{$6673$} & \scalebox{1}{$31971$}\\\hline
     2 & & \scalebox{1}{$51$} & \scalebox{1}{$217$} & \scalebox{1}{$1001$} & \scalebox{1}{$4863$} & \scalebox{1}{$22,265$} & \scalebox{1}{$106,153$}\\\hline
     3 & & & \scalebox{1}{$846$} & \scalebox{1}{$3673$} & \scalebox{1}{$16,693$} & \scalebox{1}{$77,064$} & \scalebox{1}{$360,517$} \\\hline
     4 & & & & \scalebox{1}{$14,637$} & \scalebox{1}{$62,221$} & \scalebox{1}{$282,534$} & \scalebox{1}{$1,285,131$} \\\hline
     5 & & & & & \scalebox{1}{$247,965$} & \scalebox{1}{$1,054,081$} & \scalebox{1}{$4,778,353$} \\\hline
     6 & & & & & & \scalebox{1}{$4,200,768$} & \scalebox{1}{$17,857,153$}\\\hline
     7&&&&&&& \scalebox{1}{$71,165,091$} \\\hline
\end{tabular}
\caption{Numbers  $m_{(q,p)}$ for small values of $p \leq q \leq 7$.  When $\gcd(p,q) = 1$, these are generalized Markov numbers $m_{\frac{p}{q}}$.  For discussion of the other values,  see Section \ref{sec:extendalgorithm}}\label{table:GenMarkov}
\end{table}

\begin{table}
\begin{tabular}{|c||c|c|c|c|c|c|}
    \hline
    $p \backslash q$&1&2&3&4&5&6 \\ \hline \hline
    1 & [3] & [3,4]&[4,1,2,4]&[4,1,2,3,1,4] &[4,1,3,1,2,3,1,4] &[4,1,3,1,2,3,1,3,1,4]\\\hline
    2 & & [3,5,3] & [3,4,5,3] & [3,4,5,1,3,3] & [4,2,1,4,5,1,2,4] & [4,2,1,4,5,1,3,2,1,4]\\\hline
    3 & & & [3,5,3,5,3] & [3,5,3,4,5,3] & [3,4,5,1,2,5,4,3] & [3,4,5,1,2,4,5,1,2,4]\\\hline
    4 &&&& [3,5,3,5,3,5,3] & [3,5,3,4,5,3,5,3] & [3,4,5,3,5,1,2,5,4,3]\\\hline
    5 &&&&& [3,5,3,5,3,5,3,5,3]&[3,5,3,5,3,4,5,3,5,3] \\\hline
    6 &&&&&&[3,5,3,5,3,5,3,5,3,5,3] \\\hline
\end{tabular}
\caption{Continued fractions given from sign sequences $\mathbf{f}(q,p)$.  In each case, $m_{(q,p)}$ is the numerator of this continued fraction.}\label{table:contfrac}
\end{table} 
 
 For any $\frac{p}{q}$, the sequence $\mathbf{f}(\frac{p}{q}) = (f_0,f_1,\ldots,f_{4(p+q) - 6})$ is anti-symmetric. That is, for $i < 2(p+q) - 3$, $f_{i} = - f_{(4(p+q)-6)-i}$. Moreover, the middle term of the sequence $\mathbf{f}(\frac{p}{q})$, $f_{2(p+q) - 3}$, always corresponds to a crossing which is the midpoint of the segment of $\sigma_{p+q-1}$; no other crossing can occur at a midpoint, or else $\gamma_{\frac{p}{q}}$ would go through other integral points besides $(0,0)$ and $(q,p)$. By Lemma \ref{lem:FlipContFraction}, since the rest of the sequence is anti-symmetric, the choice of sign at the midpoint does not matter. 

In Theorem \ref{thm:GenMarkovContFrac}, we show that $\mathbf{f}(\frac{p}{q})$ exactly gives the sign sequence for the snake graph $G_{\gamma,T}$ which encodes the generalized Markov number $m_{\frac{p}{q}}$. Recall such an arc $\gamma$ and snake graph $G_{\gamma,T}$ is guaranteed in Corollary \ref{cor:Gyoda}.

\begin{theorem}\label{thm:GenMarkovContFrac}
Let $p,q \in \mathbb{Z}$ be such that $p < q$ and $\gcd(p,q) = 1$. Let $\mathbf{f}(\frac{p}{q})$ be the sign function on $\frac{p}{q}$, and let $[a_1,\ldots,a_m]$ be the continued fraction from $\mathbf{f}(\frac{p}{q})$. Then, if $m_{\frac{p}{q}}$ is the generalized Markov number associated to $\frac{p}{q}$,\[
m_{\frac{p}{q}} = \mathcal{N}([a_1,\ldots,a_m]) = m(G_{\overline{\gamma}_{\frac{p}{q}},T_0}),
\]
where $\overline{\gamma}_{\frac{p}{q}}$ is the arc on $\mathcal{O}_3$ associated to the generalized Markov number $m_{\frac{p}{q}}$ and $T_0$ is the initial triangulation of $\mathcal{O}_3$. 
\end{theorem}

\begin{proof}

We will induct on the number of flips to reach the arc with slope $\frac{p}{q}$.  Once we show that $\mathbf{f}(\frac{p}{q})$ is the same as the sign sequence of $G_{\overline{\gamma}_{\frac{p}{q}},T_0}$, the statement of the theorem will follow from Theorem \ref{thm:ContFracSnakeGraph}.  

Since flipping arcs is an involution, we update how we go between rational numbers, as in Lemma \ref{lem:SlopeArcs}, to also make it an involution.  Let $(\frac{a}{b}, \frac{c}{d}, \frac{e}{f}) \in \Q^3$.  We define the \emph{mutation} of $\frac{e}{f}$ as $\mu_{\frac{e}{f}}(\frac{a}{b}, \frac{c}{d}, \frac{e}{f}) = (\frac{a}{b}, \frac{c}{d}, \frac{e'}{f'})$ where \[
\frac{e'}{f'} = \begin{cases}
\frac{a+c}{b+d} & \frac{e}{f} \neq \frac{a+c}{b+d}\\
\frac{c-a}{d-b} & \frac{e}{f} =\frac{a+c}{b+d}.\\
\end{cases}
\]

In the following, we analyze arcs on the orbifold $\mathcal{O}_3$, but we retain our labeling of arcs with rational numbers. That is, we label the arcs in  $T_0$ as $(\tau_{\frac{0}{1}},\tau_{\frac{1}{0}},\tau_{\frac{-1}{1}})$ and use the convention that in a sequence of flips,  the first arc flipped is $\tau_{\frac{-1}{1}}$ and, if the sequence is at least length two, the next flip is at $\tau_{\frac{1}{0}}$. In general, the flip of $\tau_{\frac{e}{f}}$ in a triangulation $(\tau_{\frac{a}{b}}, \tau_{\frac{c}{d}},\tau_{\frac{e}{f}})$ results in the triangulation $(\tau_{\frac{a}{b}}, \tau_{\frac{c}{d}},\tau_{\frac{e'}{f'}})$ where $\frac{e'}{f'} = \mu_{\frac{e}{f}}(\frac{a}{b}, \frac{c}{d}, \frac{e}{f})$.

For our base cases, one can check the claim for slopes $\frac{1}{1}$ and $\frac{1}{2}$ directly. Note these correspond to one and two (distinct) flips from $T_0$ respectively.

 Let $\mathbf{\alpha} = \alpha_1,\ldots,\alpha_\ell$ be a sequence of rational numbers such that $\alpha_1 = \frac{-1}{1},  \alpha_2 = \frac{1}{0}$,  and for all $2 < i \leq \ell$, $\alpha_i \in \mu_{\alpha_{i-1}} \circ \cdots \circ \mu_{\alpha_1}(\frac{0}{1}, \frac{1}{0},\frac{-1}{1})$.  We let $\mu_{\mathbf{\alpha}}$ denote the composition $\mu_{\alpha_\ell} \circ \cdots \circ \mu_{\alpha_1}$.  By Lemma \ref{lem:SlopeArcs},  as long as our mutation sequence has length at least one,  we know that  the tuple $\mu_{\mathbf{\alpha}}(\frac01, \frac10,\frac{-1}{1})$ is of the form  $(\frac{a}{b},\frac{c}{d},\frac{a+c}{b+d})$ where $c \geq a$ and $d \geq b$.  Since these are distinct numbers,  at least one inequality is strict. 
 
 We show that  knowing $\mathbf{f}(\frac{c}{d})$ and the sign sequence of $G_{\overline{\tau_{\frac{c}{d}}}, T_0}$ are equal will also tell us that  $\mathbf{f}(\frac{a+c}{b+d})$ and the sign sequence of $G_{\overline{\tau_{\frac{a+c}{b+d}}}, T_0}$ are equal. In the following, we drop the overlines for arcs in $\mathcal{O}_3$.  First, consider a triangulation of $\mathcal{O}_3$, $(\tau_{\frac{a}{b}}, \tau_{\frac{c}{d}}, \tau_{\frac{a+c}{b+d}})$. Since these three arcs are compatible, we can see that there is an initial section of $\tau_{\frac{a+c}{b+d}}$ which is homotopic to  $\tau_{\frac{c}{d}}$. Suppose that $\tau_{\frac{c}{d}}$ crosses all three arcs from $T_0$ at least once. Then, the section of $\tau_{\frac{a+c}{b+d}}$ which is homotopic to $\tau_{\frac{c}{d}}$ is over halfway along the arc; note that since pending arcs are loops,  they all have a ``halfway'' point where they curve around the orbifold point that they enclose. Thus, we know that the first $4(c+d) - 5$ entries of the sign sequence for $G_{\tau_{\frac{a+c}{b+d}},T_0}$ and $G_{\tau_{\frac{c}{d}},T_0}$ are the same. By the symmetry of the snake graph for a pending arc (given by the fact that pending arcs are drawn as loops), this completely determines the sign sequence for $G_{\tau_{\frac{a+c}{b+d}},T_0}$. 
 
 Since we already considered the base cases of slopes $\frac11$ and $\frac12$, we do not need to consider the case when $\tau_{\frac{c}{d}}$ only crosses one arc from $T_0$.  So we next suppose that $\tau_{\frac{c}{d}}$ only crosses two of the arcs from $T_0$. If $\tau_{\frac{a+c}{b+d}}$ also only crosses two arcs from $T_0$, then we have the same situation as above. Note in this case that $\frac{a}{b} = \frac{0}{1}$ since we are assuming we only flipped two of the arcs in the initial triangulation,  which by convention are $\tau_{\frac{-1}{1}}$ and $\tau_{\frac{1}{0}}$ . Now suppose that $\tau_{\frac{a+c}{b+d}}$ does cross all three arcs from $T_0$; necessarily,  $\tau_{\frac{a+c}{b+d}}$ is the result of flipping $\tau_{\frac{0}{1}}$ in the tuple $(\tau_{\frac{0}{1}}, \tau_{\frac{1}{n}}, \tau_{\frac{1}{n+1}})$; i.e., $a=c=1,  b = n \geq 1, d = n+1$.  Then the segment of $\tau_{\frac{a+c}{b+d}}$ which is homotopic to $\tau_{\frac{c}{d}}$ is the portion until the first intersection of $\tau_{\frac{a+c}{b+d}}$ and $\tau_{\frac{0}{1}}$. Note that the halfway point of $\tau_{\frac{a+c}{b+d}}$ occurs between its two intersections with $\tau_{\frac{0}{1}}$. In order to have our conventions agree, suppose that $\tau_{\frac{1}{0}}$ follows $\tau_{\frac{-1}{1}}$ in clockwise order, so that $\tau_{\frac{-1}{1}}$ follows $\tau_{\frac{0}{1}}$ in clockwise order. Then, the first part of the sign sequence for $G_{\tau_{\frac{a+c}{b+d}},T_0}$ consists of the sign sequence for $G_{\tau_{\frac{c}{d}},T_0}$ followed by $+$ since the last arc which $\tau_{\frac{c}{d}}$ crosses is $\tau_{\frac{-1}{1}}$, and the next arc $\tau_{\frac{a+c}{b+d}}$ crosses, after the section which is homotopic to $\tau_{\frac{c}{d}}$, is $\tau_{\frac{0}{1}}$.  The next sign is the middle sign of the sign sequence for $G_{\tau_{\frac{a+c}{b+d}},T_0}$ which depends on the orientation of $\tau_{\frac{a+c}{b+d}}$. All other signs are determined by the symmetry of sign sequences for snake graphs from pending arcs. 
 
 Now we consider how the sequences $\mathbf{f}(\frac{c}{d})$ and $\mathbf{f}(\frac{a+c}{b+d})$ compare and show it is identical to the case for the snake graphs from the arcs with these labels. 
For $a \leq c, b \leq d$, we have $\lfloor \frac{4(a+b+c+d) - 5}{2}\rfloor = 2(a+b+c+d) - 2 < 2(2c + 2d - 1) -2 = 4(c+d) - 4$.  By induction we can show that the triangle with vertices $(0,0), (d+b,a+c),(d,c)$ has no interior vertices in $\Z^2$; from this, we know the first $4(c+d) - 5$ entries of $\mathbf{f}(\frac{a+c}{b+d})$ are the same as the entries of $\mathbf{f}(\frac{c}{d})$.  When $c>a$ and $d>b$, this covers over half the sequence $\mathbf{f}(\frac{a+c}{b+d})$; by the anti-symmetry of the sign sequences, this completely determines $\mathbf{f}(\frac{a+c}{b+d})$.  Since we knew that $\mathbf{f}(\frac{c}{d})$ matched the sign sequence for $G_{\tau_{\frac{c}{d}},T_0}$, we know the same for $\mathbf{f}(\frac{a+c}{b+d})$ and $G_{\tau_{\frac{a+c}{b+d}},T_0}$
 
 The case when we do not have both $c>a$ and $d>b$ is again when $\frac{a}{b} = \frac{1}{n}$ and $\frac{c}{d} = \frac{1}{n+1}$. In this case,  if $\mathbf{f}(\frac{a+c}{b+d}) = \mathbf{f}(\frac{2}{2n+1})= (f_0,\ldots,f_{4(3+2n)-6})$,  the first $4(n+2)-6 = 4n+2$ entries are the same as $f(\frac{1}{n+1})$.  We know that $f_{4n+2} = +$ since this corresponds to the shared vertex between a diagonal and horizontal line segment, when traveling from the diagonal crossing to the horizontal.  Then,  $f_{4n+3}$ can be either $+$ or $-$, as this crossing occurs at the midpoint of $\sigma_{p+q-1}$. The rest of the entries follow from the anti-symmetry. We can again see that, since the sign sequence for $G_{\tau_{\frac{c}{d}},T_0}$ agrees with $\mathbf{f}(\frac{c}{d})$, the sequence for $G_{\tau_{\frac{a+c}{b+d}},T_0}$ will agree with $\mathbf{f}(\frac{a+c}{b+d})$, as long as we choose the same sign at the middle term. (Moreover, this middle term will not affect the number of perfect matchings of $G_{\tau_{\frac{a+c}{b+d}},T_0}$ since it does not affect the numerator of the corresponding continued fraction by Lemma \ref{lem:FlipContFraction}.) 
\end{proof}

We note that a version of Theorem \ref{thm:GenMarkovContFrac} is already known for ordinary Markov numbers. Given $p<q$ with $\gcd(p,q) = 1$,  let $\mathbf{e}(\frac{p}{q})$ be the sequence of length $2(p+q)-2$ which consists only of the even-indexed entries from $\mathbf{f}(\frac{p}{q})$; that is, $\mathbf{e}(\frac{p}{q}) = (f_0,f_2,\ldots, f_{4(p+q)-6})$. This sequence was described in \cite{CS}, but the result was originally known by Frobenius. 

\begin{theorem}[\cite{frobenius1913markoffschen},\cite{CS}]\label{thm:OrdMarkovContFrac}
Let $p,q \in \mathbb{Z}$ be such that $p < q$ and $\gcd(p,q) = 1$. Let $\mathbf{e}(\frac{p}{q})$ be the even-indexed sign function on $\frac{p}{q}$, and let $[b_1,\ldots,b_m]$ be the continued fraction from $\mathbf{e}(\frac{p}{q})$. Then, if $n_{\frac{p}{q}}$ is the ordinary Markov number associated to $\frac{p}{q}$, \[
n_{\frac{p}{q}} = \mathcal{N}([b_1,\ldots,b_m]).
\]
\end{theorem}

We have thus far considered arcs $\gamma_{\frac{p}{q}}$ with $p\leq q$  so that each generalized Markov number corresponds to a unique arc. However, we could equivalently consider arcs between the origin and $(p,q)$; this would correspond to flip sequences which begin with flipping first $\tau_{\frac{-1}{1}}$ and then $\tau_{\frac{0}{1}}$. By our description of the meaning $\mathbf{f}$ at the beginning of the section, we can extend the domain of $\mathbf{f}$ to include all $\mathbb{Q}_{>0}$. To line up the conventions, we say that, if $p<q$, then $f_0 = +$ and $f_{4(p+q)-6} = -$ in $\mathbf{f}(\frac{q}{p})$.

\begin{prop}\label{prop:antisymmetry}
Let $p < q$ and $\gcd(p,q) = 1$. Suppose $\mathbf{f}(\frac{p}{q}) = (f_0,\ldots,f_{4(p+q)-6})$ and $\mathbf{f}(\frac{q}{p}) = (f_0',\ldots,f_{4(p+q)-6}')$. Then for all $0 \leq i \leq 4(p+q)-6$, besides possibly $i = 2(p+q)-3$, $f_i = -$ if and only if $f_i' = +$. 
\end{prop}

\begin{proof}
One can check that at each stage of the algorithm, the signs will be reversed for slopes $\frac{p}{q}$ and $\frac{q}{p}$. We can again choose either sign for $f_{2(p+q)-3}$ and $f'_{2(p+q)-3}$. 
\end{proof}

\section{Patterns}\label{sec:Patterns}

In this section, we explore patterns amongst the generalized Markov numbers $m_{\frac{p}{q}}$ and the continued fraction expansions given in Section \ref{sec:algorithm}.  Given integers $p<q$ with $\gcd(p,q) = 1$, let $C_{\frac{p}{q}} = a_1,\ldots,a_m$ be the sequence of positive integers such that $m_{\frac{p}{q}} = \mathcal{N}([C_{\frac{p}{q}}])$, subject to our conventions.  
Similarly,  let $C^{\text{ord}}_{\frac{p}{q}} = b_1,\ldots,b_{m'}$ be the sequence of positive integers such that $n_{\frac{p}{q}} = \mathcal{N}[C^{\text{ord}}_{\frac{p}{q}}]$,  with the same conventions.  
We begin this section by describing some properties of the sequence $C_{\frac{p}{q}}$ and how it compares to $C^{\text{ord}}_{\frac{p}{q}}$.  Then, we analyze the behavior of two special families of generalized Markov numbers, $m_{\frac{1}{q}}$ and $m_{\frac{q-1}{q}}$, and compare these with the corresponding families in the ordinary Markov case. 

\subsection{Properties of $C_{\frac{p}{q}}$}

In this subsection, we explain more concretely how the continued fractions associated to generalized and ordinary Markov numbers compare.  This allows us to give an elementary description of the continued fractions produced by the algorithm given in Section \ref{sec:algorithm}.

\begin{lemma}\label{lem:compareGenOrd}
Let $p < q$ be such that $\gcd(p,q) = 1$.  Let $C_{\frac{p}{q}} = a_1,\ldots,a_{m_1}$,  and $C^{\text{ord}}_{\frac{p}{q}} = b_1,\ldots,b_{m_2}$ .  Then, $m_1 = m_2$ and \begin{enumerate}
\item if $b_i = 1,  a_i \in \{1,2,3\}$, and 
\item if $b_i = 2,  a_i \in \{3,4,5\}$.
\end{enumerate}
\end{lemma}

\begin{proof}
The fact $m_1 = m_2$ is implied by the observation that there will not be any more sign changes amongst the full sequence $f_0,f_1\ldots f_{4(p+q - 7}, ,f_{4(p+q)-6}$ than if we just consider the even-indexed terms $f_0,f_2,\ldots f_{4(p+q)-8},f_{4(p+q)-6}$.  
Consider two consecutive even-indexed terms which are equal, $f_{2i}=f_{2i+2}$; this occurs if $\gamma_{\frac{p}{q}}$ is crossing three arcs, say $\sigma_i, \sigma_{i+1},\sigma_{i+2}$ which share an endpoint.  Then, this shared endpoint is also the endpoint of $\sigma_{i+1}$ which is closer to $\gamma_{\frac{p}{q}}$,  guaranteeing that $f_{2i} = f_{2i+1} = f_{2i+2}$.

The other statements follow quickly.  If $f_{2i-2} \neq f_{2i}$ and $f_{2i} \neq f_{2i+2}$, the largest subsequence of the same sign which includes $f_{2i}$ has length 3.  Similarly,  if $f_{2i-2} = f_{2i}$,  which necessarily means $f_{2i-4} \neq f_{2i-2}$ and $f_{2i} \neq f_{2i+2}$, the largest subsequence of the same sign which includes $f_{2i}$ has length 5.
\end{proof}

We can combine this lemma with the following result given by Frobenius,  and reproven in \cite{CS} using snake graphs,  to give a similar description in our generalized case.

\begin{theorem}\cite{frobenius1913markoffschen}\label{thm:Frobenius}
Let $p < q$ be such that $\gcd(p,q) = 1$.  Let $C^{\text{ord}}_{\frac{p}{q}} = b_1,\ldots,b_m$.  Then,  $b_i \in \{1,2\}$ for all $i$,  $m$ is necessarily even, and  $b_i = b_{m-i+1}$. 
\begin{enumerate}
\item If $p+1 = q$,  then each $b_i = 2$ and $m = 2p$.
\item If $p+1 < q$,  there exists a unique positive integer $c$ satisfying $\frac{c-1}{c} < \frac{p}{q} < \frac{c}{c+1}$, and  \begin{enumerate}
\item there are at most $p+1$ subsequences of 2s; the first and last are of length $2c-1$ and all others are of length $2c$ or $2c+2$; 
\item there are at most $p$ subsequences of 1s, with the $i$-th subsequence having length $2\mu_i$, where the  $\mu_i$ satisfy $\vert \mu_i - \mu_j \vert \leq 1$ for all $i,j$.  
\end{enumerate}
\end{enumerate}
\end{theorem}

\begin{cor}\label{cor:elementaryDescription}
Let $p < q$ be such that $\gcd(p,q) = 1$.  Let $C_{\frac{p}{q}} = a_1,\ldots,a_m$  with $a_1,a_m > 1$.  Then,\begin{enumerate}
\item $1 \leq a_i \leq 5$;
\item $m = 2(q-1)$;
\item if $i < q-1$,  $a_i = a_{2(q-1) - i + 1}$,  and $\vert a_{q-1} - a_{q} \vert = 1$;
\item if $p+1 = q$,  each $a_i \in \{3,4,5\}$;
\item if $p+1 < q$,  let $c$ be as in Theorem \ref{thm:Frobenius}. \begin{enumerate}
\item There are at most $p+1$ subsequences of numbers in $\{3,4,5\}$ which are not all 3; the first and last have length $2c-1$ and all others are of length $2c$ or $2c+2$.
\item There are at most $p$ subsequences of numbers in $\{1,2,3\}$ which are not all 3,  with the $i$-th subsequence having length $2\mu_i$ where the $\mu_i$ satisfy $\vert \mu_i - \mu_j \vert \leq 1$ for all $i,j$. 
\end{enumerate}
\end{enumerate}
\end{cor}

\begin{proof}
Most parts of the result follow from combining Lemma \ref{lem:compareGenOrd} and Theorem \ref{thm:Frobenius}.  The fact that $m = 2(q-1)$ follows from Lemma \ref{lem:compareGenOrd} and the description of the continued fractions given in \cite{CS}.  Part 3 follows from the antisymmetry of the terms. The fact that the subsequences in part 5 each contain numbers other than 3 follows from noticing that there is no configuration which would produce a consecutive pair of entries $3$.  
\end{proof}

 Part 3 of Corollary \ref{cor:elementaryDescription} shows that the middle terms in $C_{\frac{p}{q}}$ will always differ by one.  Here,  we show any such pair of adjacent integers, with both being between one and five, is attainable by analyzing the middle terms.

\begin{lemma}\label{lem:MidPointPattern}
Let $p,q \in \mathbf{Z}_{>0}$ be such that $p<q$ and $\gcd(p,q) = 1$. Let $G_{\frac{p}{q}} = \mathcal{G}[a_1,\ldots,a_{2(q-1)}]$. 
\begin{enumerate}
    \item If $p$ and $q$ are odd, then $\{a_{q-1},a_q\} = \{1,2\}$. 
    \item If  $p$ is even and $q$ is odd, hen, $\{a_{q-1},a_q\} = \{4,5\}$.
    \item If $p$ is odd, $q$ is even, and $p < 2q$, then $\{a_{q-1},a_q\} = \{2,3\}$.  
    \item If $p$ is  odd, $q$ is even, and $p > 2q$, then, $\{a_{q-1},a_q\} = \{3,4\}$.
\end{enumerate}
\end{lemma}

\begin{proof}

For each part, the statement follows from analyzing the local configuration of the arcs around the central crossing of $\gamma_{\frac{p}{q}}$. If $p$ and $q$ are both odd, so that the central crossing is a diagonal line segment which lies on $y = -x + \frac{p+q}{2}$  and if $\star$ denotes the central crossing, then we see the sign sequence near $\star$ is $-+\star -+$ where, as usual, $-$ represents shared endpoints or crossing points to the right and $+$ to the left .
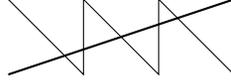
\begin{center}
\begin{tikzpicture}
\draw (0,1) -- (1,0) -- (1,1) -- (2,0) -- (2,1) -- (3,0);
\draw[thick] (3,1) -- (0, 0);
\end{tikzpicture}    
\end{center}

If $p$ is even and $q$ is odd, so that the central arc crossed is horizontal, then the local configuration is as below, and the sign sequence near $\star$ must be $+---- \star ++++-$. 

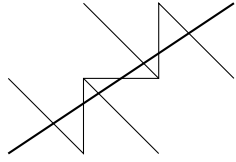
\begin{center}
\begin{tikzpicture}
\draw (0,1) -- (1,0) -- (1,1) -- (2,1) -- (2,2) -- (3,1);
\draw (1,1) -- (2,0);
\draw (1,2) -- (2,1);
\draw[thick] (0,0) -- (3,2);
\end{tikzpicture}    
\end{center}

If $p$ is odd and $q$ is even, then there are two cases for the type of local configuration around the central arc crossed. First, suppose $\frac{p}{q} < \frac12$. Then, the local configuration includes only vertical and diagonal arcs. Here, the sign sequence near $\star$ is $-++\star --+$.

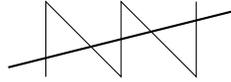
\begin{center}
\begin{tikzpicture}
\draw (0,0) -- (0,1) -- (1,0) -- (1,1) -- (2,0) -- (2,1);
\draw[thick] (-.5,.125) -- (2.5,.875);
\end{tikzpicture}    
\end{center}

If $\frac{p}{q} > \frac12$, then there are horizontal arcs near the central arc crossed, and the sign sequence near $\star$ is $-+++\star ++++-$.

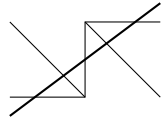
\begin{center}
\begin{tikzpicture}
\draw (0,0) -- (1,0) -- (1,1) -- (2,1);
\draw (1,0) -- (0,1);
\draw (1,1) -- (2,0);
\draw[thick] (2,1.25) -- (0,-.25);
\end{tikzpicture}    
\end{center}

\end{proof}

\subsection{Generalized Markov numbers for $\frac{1}{q}$}\label{subsec:1q}

For ordinary Markov numbers, the sequence $\{n_{\frac{1}{q}}\}_q$ is every other Fibonacci number: $n_{\frac11} = 2, n_{\frac12} = 5, n_{\frac13} = 13,$ and so on. Here we study $\{m_{\frac{1}{q}}\}_q$. 

In the following, let $\alpha = (1,3)$ and let $\alpha^{-1} = (3,1)$.  

\begin{lemma}\label{lem:Pattern1q}
Let $q \geq 3$. If $q$ is odd, then $C_{\frac{1}{q}} = [4,\alpha^{\frac{q-3}{2}},2,1,(\alpha^{-1})^{\frac{q-3}{2}},4]$. If $q$ is even, then $C_{\frac{1}{q}} = [4,\alpha^{\frac{q-4}{2}},1,2,3,1,(\alpha^{-1})^{\frac{q-4}{2}},4]$.
\end{lemma}

\begin{proof}
When $q >> 0$,  we can see a pattern when we are far from the middle crossing point.  To the left of the middle crossing point, then $\gamma_{\frac{1}{q}}$ will cross each arc closer to its right endpoint. The shared endpoints alternate between right and left since $\gamma_{\frac{1}{q}}$ only crosses diagonal and vertical line segments.  The same is true to the right of the middle crossing point, except that $\gamma_{\frac{1}{q}}$ now crosses all arcs closer to their right endpoint.  The first and last entries are 4 because of our convention to add an extra $-$ at the beginning and an extra $+$ at the end of our sign sequence.  

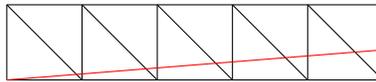
\begin{center}
\begin{tikzpicture}
\draw (0,0) -- (0,1);
\draw (1,0) -- (1,1);
\draw (2,0) -- (2,1);
\draw (3,0) -- (3,1);
\draw (4,0) -- (4,1);
\draw (5,0) -- (5,1);
\draw (0,0) -- (5,0);
\draw (0,1) -- (5,1);
\draw (1,0) -- (0, 1);
\draw (2,0) -- (1, 1);
\draw (3,0) -- (2, 1);
\draw (4,0) -- (3, 1);
\draw (5,0) -- (4, 1);
\draw[red] (0,0) -- (5,.4);
\end{tikzpicture}
\end{center}

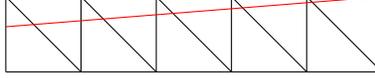
\begin{center}

\begin{tikzpicture}
\draw (0,0) -- (0,1);
\draw (1,0) -- (1,1);
\draw (2,0) -- (2,1);
\draw (3,0) -- (3,1);
\draw (4,0) -- (4,1);
\draw (5,0) -- (5,1);
\draw (0,0) -- (5,0);
\draw (0,1) -- (5,1);
\draw (1,0) -- (0, 1);
\draw (2,0) -- (1, 1);
\draw (3,0) -- (2, 1);
\draw (4,0) -- (3, 1);
\draw (5,0) -- (4, 1);
\draw[red] (0,.6) -- (5,1);
\end{tikzpicture}
\end{center}

Next we check how the pattern deviates near the middle crossing.  As discussed in Lemma \ref{lem:MidPointPattern},  this will depend on the parity of $q$.  If $q$ is odd,  we can see that after three entries $-$, $\gamma_{\frac{1}{q}}$ enters the quadrilateral around the diagonal line segment which is the central arc crossed.  Then, choosing $+$ at the middle crossing point, we have entries $++,-$.

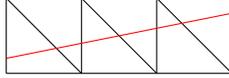
\begin{center}
\begin{tikzpicture}
\draw (0,0) -- (0,1);
\draw (1,0) -- (1,1);
\draw (2,0) -- (2,1);
\draw (3,0) -- (3,1);
\draw (0,0) -- (3,0);
\draw (0,1) -- (3,1);
\draw (1,0) -- (0, 1);
\draw (2,0) -- (1, 1);
\draw (3,0) -- (2, 1);
\draw[red] (0,.2) -- (3,.8);
\end{tikzpicture}
\end{center}

If $q$ is even,  then the pattern is interrupted at the middle crossing point in a different way.  After a single entry $+$, $\gamma_{\frac{1}{q}}$ enters the quadrilateral around the vertical line segment which is the central arc crossed; here,  we have entries $+++, --$.

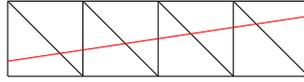
\begin{center}
\begin{tikzpicture}
\draw (0,0) -- (0,1);
\draw (1,0) -- (1,1);
\draw (2,0) -- (2,1);
\draw (3,0) -- (3,1);
\draw (4,0) -- (4,1);
\draw (0,0) -- (4,0);
\draw (0,1) -- (4,1);
\draw (1,0) -- (0, 1);
\draw (2,0) -- (1, 1);
\draw (3,0) -- (2, 1);
\draw (4,0) -- (3, 1);
\draw[red] (0,.2) -- (4,.8);
\end{tikzpicture}
\end{center}

\end{proof}

Knowing the general pattern of $C_{\frac{1}{q}}$ allows us to compute the limit of $\frac{m_{\frac{1}{q}}}{m_{\frac{1}{q-1}}}$ as $q \to \infty$.  In this section and the next, we will compute several infinite periodic (or almost periodic) continued fractions. Since several computations are very similar, we introduce a unifying lemma for computing nearly 2-periodic infinite continued fractions.

\begin{lemma}\label{lem:TwoPeriodicCont}
Let $z_1,z_2,r$ be nonzero elements of a ring $R$. Then, if the continued fraction $[z_1,z_2,rz_1,r^{-1}z_2,r^2z_1,r^{-2}z_2,\ldots]$ converges, it converges to  \[
\frac{r(1+z_1z_2) - 1 \pm \sqrt{(1-r(1 + z_1z_2))^2 +4rz_1z_2}}{2rz_2}.
\]
for some choice of sign.
\end{lemma}

\begin{proof}
Let $L = [z_1,z_2,rz_1,r^{-1}z_2,r^2z_1,r^{-2}z_2,\ldots]$, assuming $z_1,z_2,r$ are chosen such that the infinite continued fraction converges. The proof follows from the fact that $[ra_1,r^{-1}a_2,ra_3,\ldots,r^{(-1)^{m-1}}a_{m}] = r[a_1,a_2,\ldots,a_m]$; this is for example given as Lemma 2.3 in \cite{CS3}. By taking the limit of this fact to an infinite continued fraction, we see that $[rz_1,r^{-1}z_2,r^2z_1,r^{-2}z_2,\ldots] = rL$. Then, $L$ satisfies the following, \[
L = z_1 + \frac{1}{z_2 + \frac{1}{rL}},
\]
which, by simplifying, implies that $L$ satisfies the following quadratic equation,\[
rz_2L^2 + (1-r(1+z_1z_2))L - z_1 = 0.
\]
The statement follows from solving this quadratic equation.
\end{proof}

The next result is our first application for Lemma \ref{lem:TwoPeriodicCont}.

\begin{prop}\label{prop:Limit1/q}
\[
\lim_{q \to \infty} \frac{m_{\frac{1}{q}}}{m_{\frac{1}{q-1}}} = \frac{5 + \sqrt{21}}{2}
\]
\end{prop}
\begin{proof}

Let $q \geq 3$.  First, suppose that $q$ is odd.  Then, by definition we have \[
[4, (\alpha)^\frac{q-3}{2}, 2,1,(\alpha^{-1})^{\frac{q-3}{2}},4] = \frac{\mathcal{N}[4, (\alpha)^\frac{q-3}{2}, 2,1,(\alpha^{-1})^{\frac{q-3}{2}},4] }{\mathcal{N}[(\alpha)^\frac{q-3}{2}, 2,1,(\alpha^{-1})^{\frac{q-3}{2}},4]} .
\]

We can further manipulate the denominator, using Lemma \ref{lem:OneAtEnd}, \begin{align*}
\mathcal{N}[(\alpha)^\frac{q-3}{2}, 2,1,(\alpha^{-1})^{\frac{q-3}{2}},4] = \mathcal{N}[(4, (\alpha)^{\frac{q-5}{2}}, 2,1,(\alpha^{-1})^{\frac{q-3}{2}},4] \\= \mathcal{N}[(4, (\alpha)^{\frac{q-5}{2}}, 2,1,3,1(\alpha^{-1})^{\frac{q-5}{2}},4] = m_{\frac{1}{q-1}}
\end{align*}

Thus, for odd $q$,  \[
[4,(\alpha)^{\frac{q-3}{2}},2,1,  (\alpha^{-1})^{\frac{q-3}{2}},4] = \frac{m_{\frac{1}{q}}}{m_{\frac{1}{q-1}}}.
\]

A similar computation shows the same is true for even $q \geq 4$ and the continued fraction $[4, (\alpha)^{\frac{q-4}{2}},2,1,3,1,(\alpha^{-1})^{\frac{q-4}{2}},4]$.  Thus,  we can express the limit of the ratio of $m_{\frac{1}{q}}$ and $m_{\frac{1}{q-1}}$ as a periodic, infinite continued fraction,  \[
\lim_{q \to \infty} \frac{m_{\frac{1}{q}}}{m_{\frac{1}{q-1}}} = [4, \overline{1,3}].
\]

With Lemma \ref{lem:TwoPeriodicCont},  using $r = 1, z_1 = 1, $ and $z_2 = 3$,  we compute $ [\overline{1,3}] = \frac{3 + \sqrt{21}}{6}$.  Then,  
 \[
\lim_{q \to \infty} \frac{m_{\frac{1}{q}}}{m_{\frac{1}{q-1}}}  = [4, \overline{1,3}] = 4 + \frac{6}{3 + \sqrt{21}} = \frac{5 + \sqrt{21}}{2}
\]

\end{proof}

The corresponding limit for ordinary Markov numbers is well-known   \[
\lim_{q \to \infty} \frac{n_{\frac{1}{q}}}{n_{\frac{1}{q-1}}} \to \frac{3 + \sqrt{5}}{2} = 1 + \varphi.
\]

Here we give a cluster-algebraic interpretation to the limit given in Proposition \ref{prop:Limit1/q}.  For $q \geq 1$, we add labels to the edges on the snake graph $\mathcal{G}[a_1,\ldots,a_m]$ where $C_{\frac{1}{q}} = a_1,\ldots,a_m$, thereby producing the snake graph $G_{\frac{1}{q}}$ for the arc $\overline{\gamma_{\frac{1}{q}}}$ on $\mathcal{O}_3$ associated to $m_{\frac{1}{q}}$. Then, we set $x_{\frac{1}{q}}^{T_0} = \frac{1}{\text{cross}(\gamma_{\frac{1}{q}},T_0)} \sum_P x(P)$, using notation from Theorem \ref{thm:SnakeGraphSurface}. Note that $\text{cross}(\gamma_{\frac{1}{q}},T_0) =  x_{\frac{-1}{1}}^{2q} x_{\frac{1}{0}}^{2(q-1)}$, where we continue our labeling scheme for the three arcs in $T_0$.

In \cite{CS3}, the authors give a way to compute the quantity $\chi(G_{\gamma,T}):= \frac{1}{\text{cross}(\gamma,T)} \sum_P x(P)$ via continued fractions. Given the snake graph $G_{\gamma,T}$, the authors define a family of Laurent polynomials $L_1,\ldots,L_m$ such that $x_\gamma = [L_1,\ldots,L_m]$. 
These are given by considering certain subgraphs of $G_{\gamma,T}$ determined by the sign sequence. Let $a_1,\ldots,a_m$ be such that the shape of $G_{\gamma,T}$ is $\mathcal{G}[a_1,\ldots,a_m]$, and suppose $G_{\gamma,T}$ has $d$ tiles. For $1 \leq i \leq m-1$, let $\ell_i = \sum_{j=1}^i a_j$; for convenience, set $\ell_0 = 0$ and $\ell_m = d+1$. Then for $1 \leq i \leq m$, if $a_i > 1$, set $H_i = (G_{\ell_{i-1}+1},G_{\ell_{i-1}+2},\ldots,G_{\ell_i-1})$ (this is the subgraph of $G_{\gamma,T}$ given by only considering these tiles), and if $a_i = 1$, we set $H_i$ as the edge shared by tiles $G_{\ell_{i-1}}$ and $G_{\ell_i}$.   If $a_0 = 1$, we set $H_1$ as the edge where we have chosen the first sign of the sequence, and similarly for $H_m$. 

We also define a family of terms $b_i$. For $1 \leq i \leq m-1$, let $b_i$ be the label of the tile $G_{\ell_i}$. Set $b_0$ as the label of the edge in $\{S(G_1),W(G_1)\}$ which is not used in the sign sequence and choose $b_m$ similarly.  In summary,  the variables $b_i$ record the edges and tiles that we ignore when forming the subgraphs $H_j$. 

Then, the Laurent polynomials $L_i$ are defined by $L_1 = \frac{1}{b_1} \chi(H_1), L_2 = \frac{b_1}{b_0b_2} \chi(H_2)$, and for $i \geq 3$, \[
L_i = \begin{cases} \frac{b_0 b_2^2 b_4^2 \cdots b_{i-3}^2 b_{i-1}}{b_1^2b_3^2\cdots b_{i-2}^2b_i} \chi(H_i) & i \text{ is odd}\\ 
& \\
\frac{b_1^2b_3^2\cdots b_{i-3}^2b_{i-1}}{b_0 b_2^2 b_4^2 \cdots b_{i-2}^2b_i} \chi(H_i) & i \text{ is even}.
\end{cases}
\]

 In Section 7 of \cite {CS3}, the authors compute a few limits of the form $\chi(G_i)/\chi(G_{i-1})$ for snake graphs $G_i$ growing increasingly larger.  We can follow their reasoning to compute similar limits in our construction.  

We introduce the snake graph $G_{\frac{1}{q}}$ for $q \geq 1$. In this section, for simplicity in figures and calculations we set $x_{\frac{-1}{1}} = x_1, x_{\frac{1}{0}} = x_2,$ and $x_{\frac{0}{1}} = x_3$.  The snake graph $G_{\frac{1}{q}}$ has $4q-2$ tiles, and for $q>>1$ the first section is as in Figure \ref{fig:G1q}. Since we will only be concerned with limiting behavior of $\chi(G_\frac{1}{q})$, we will not give a complete example of a snake graph $G_{\frac{1}{q}}$.

\begin{figure}
\centering
\begin{tikzpicture}[scale = 1.3]
\draw (0,0) to node[below]{$x_3$} (1,0) to node[right]{$x_1$} (1,1) to node[below]{$x_1$} (2,1) -- (2,4) to node[below]{$x_1$} (3,4) to node[right]{$x_2$} (3,5) -- (1,5) -- (1,2) to node[above]{$x_2$} (0,2) -- (0,0);
\draw (0,1) to node[above]{$x_1$} (1,1) to node[right]{$x_3$} (1,2) to node[above]{$x_2$} (2,2);
\draw (1,3) to node[above]{$x_3$} (2,3);
\draw (1,4) to node[above]{$x_1$} (2,4) to node[right]{$x_3$} (2,5);
\node[right] at (2,1.5){$x_2$};
\node[right] at (2,2.5){$x_1$};
\node[right] at (2,3.5){$x_1$};
\node[above] at (2.5,5){$x_2$};
\node[above] at (1.5,5){$x_2$};
\node[left] at (1,4.5){$x_1$};
\node[left] at (1,3.5){$x_2$};
\node[left] at (1,2.5){$x_2$};
\node[left] at (0,1.5){$x_1$};
\node[left] at (0,0.5){$x_2$};
\draw[gray,dashed] (1,0) to node[right]{$x_1$}(0,1);
\draw[gray,dashed] (0,2) to node[right]{$x_1$}(1,1);
\draw[gray,dashed] (1,2) to node[right]{$x_2$}(2,1);
\draw[gray,dashed] (1,3) to node[right]{$x_2$}(2,2);
\draw[gray,dashed] (1,4) to node[right]{$x_1$}(2,3);
\draw[gray,dashed] (1,5) to node[right]{$x_1$}(2,4);
\draw[gray,dashed] (2,5) to node[right]{$x_2$}(3,4);
\node[] at (2.5,5.7){$\vdots$};
\end{tikzpicture}
\caption{First part of $G_{\frac{1}{q}}$ for $q >> 1$.}\label{fig:G1q}
\end{figure}

If $G = \mathcal{G}[a_1,\ldots,a_n]$,  let $\hat{G} = \mathcal{G}[a_2,\ldots,a_n]$; this will of course depend on our choice of sign on the first tile of $G$.  We let $\hat{G}_{\frac{1}{q}}$ be $G_{\frac{1}{q}}$ with the first (southwest-most) tile removed,  thus breaking from convention and using the expression $G_{\frac{1}{q}} = \mathcal{G}[1,3,\ldots,3,1]$. 

\begin{prop}\label{prop:InfClusterLimit}
Let $\delta = x_1+ x_2 + x_3$.
\begin{enumerate}
    \item The limit of the ratio of $\chi(G_{\frac{1}{q}})$ and $\chi(\hat{G}_{\frac{1}{q}})$ converges as $q$ goes to infinity, and this limit is equal to \[
    \frac{x_1^2 + \delta x_3 - x_2^2 + \sqrt{(x_2^2 - x_1^2 -\delta x_3)^2 + 4x_2^2x_3\delta}}{2 \delta x_2}.
    \]
    \item The limit of the ratio of $x_{\frac{1}{q}}$ and $x_{\frac{1}{q-1}}$ converges as $q$ goes to $ \infty$, and this limit is equal to \[
    \frac{\delta x_3 + x_1^2 + 2x_1x_2 - x_2^2 + \sqrt{(x_1^2 - x_2^2 - \delta x_3)^2 + 4 x_1^2x_3 \delta}}{2x_1x_2}
    \]
\end{enumerate}
\end{prop}

\begin{proof}
(1) We consider the infinite snake graph $\mathcal{G}[\overline{1,3}]$, with labels as in Figure \ref{fig:G1q}. Then, $H_{2i-1}$ for $i \geq 1$ is a single edge with label $x_3$ while $H_i$ for even $i$ is as below.
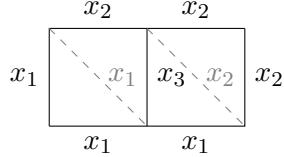
\begin{center}
    \begin{tikzpicture}[scale = 1.3]
    \draw (0,0) to node[below]{$x_1$} (1,0) to node[below]{$x_1$} (2,0) to node[right]{$x_2$} (2,1);
    \draw (2,1) to node[above]{$x_2$} (1,1) to node[above] {$x_2$} (0,1) to node[left]{$x_1$} (0,0);
    \draw (1,0) to node[right]{$x_3$} (1,1);
    \draw[gray,dashed] (0,1) to node[right]{$x_1$} (1,0);
    \draw[gray,dashed] (1,1) to node[right]{$x_2$}(2,0);
    \end{tikzpicture}
\end{center}

We have that $\chi(H_{2i-1}) = x_3$ and $\chi(H_{2i}) = \frac{x_1^2x_2 + x_1x_2^2 + x_1x_2x_3}{x_1x_2} = x_1 + x_2 + x_3 = :\delta$. Moreover, $b_{2i-1} = x_1$ and $b_{2i} = x_2$. Thus, $L_1 = \frac{x_3}{x_2}$ and $L_2 = \frac{\delta x_2}{x_1^2}$; moreover, we can see from the periodicity of the terms $\chi(H_i)$ and $b_i$ that $L_{2i+1} = \frac{x_1^2}{x_2^2} L_{2i-1}$ and $L_{2i+2} = \frac{x_2^2}{x_1^2} L_{2i}$. 

We first determine that $[L_1,L_2,\ldots]$ converges for any values $x_1,x_2,x_3 \in \mathbb{R}_{>0}$.  This uses an argument similar to Lemma 7.2 in \cite{CS3}. We know  $[L_1,L_2,\ldots,]$, with all $L_i$ evaluated at a choice of real numbers $x_1,x_2,x_3$, converges if and only if $\sum_{i \geq 1} L_i$ diverges.  We in fact can show that $\lim_{i \to \infty} L_i \neq 0$. The limit of the even-indexed terms $L_{2i}$ is given by $\lim_{i\to\infty} \bigg( \frac{x_2^2}{x_1^2} \bigg)^i \frac{\delta x_2}{x_1}$ while the limit of the odd-indexed terms is given by $\lim_{i\to\infty} \bigg( \frac{x_1^2}{x_2^2} \bigg)^i \frac{x_3}{x_2}$. It is not possible for both of these limits to converge; thus, the sequence $\{L_i\}$ diverges and $[L_1,L_2,\ldots]$ converges for any choice of positive real numbers $x_1,x_2,x_3$.

Therefore, to calculate the infinite continued expression $[L_1,L_2,L_3,\ldots]$, we can use Lemma \ref{lem:TwoPeriodicCont} with $z_1 = L_1 =\frac{x_3}{x_2}$,  $z_2 = L_2 = \frac{\delta x_2}{x_1^2}$,  and $r = \frac{x_1^2}{x_2^2}$. 
We have that $L$ equals the expression in the statement of the Proposition by noting that we assume $x_1,x_2,x_3 \in \mathbb{R}_{>0}$.  By Theorem 6.3 of \cite{CS3}, this is equal to the limit of the ratio of $\chi(G_{\frac{1}{q}})$ and $\chi(\hat{G}_{\frac{1}{q}})$.

(2) Throughout this part, we assume $q>>0$; a few claims may not hold for small $q$. If we want to consider the ratio of $x_{\frac{1}{q}} = \chi(G_{\frac{1}{q}})$ and $x_{\frac{1}{q-1}} = \chi(G_{\frac{1}{q-1}})$, we can use the same method as part (1), but swap the sign chosen on the first tile so that the associated infinite continued fraction is $[4,\overline{1,3}]$; this guarantees that the snake graph resulting from removing $H_1$ from $G_\frac{1}{q}$ is equal to the snake graph $G_{\frac{1}{q-1}}$. We will use prime marks $'$ to denote the quantities in this part,  and then we will compare the quantities to those in part (1).  We have, for $i \geq 1$,  $H_{2i}' = H_{2i+1}$ and $H_{2i+1}' = H_{2i+2}$. The subgraph $H_1'$ consists of the first three tiles of $G_{\frac{1}{q}}$, as in Figure \ref{fig:G1q}, and $\chi(H'_1) = \frac{(x_1x_2x_3 \delta + x_1^2x_2)}{x_1^2x_2} = \frac{x_3\delta + x_1x_2}{x_1}$. Moreover, we have $b'_0 = x_3$, and for $i \geq 1$,  $b'_{2i-1} = x_2$ and $b'_{2i} = x_1$. Thus, we have that $L'_1 = \frac{x_3\delta + x_1x_2}{x_1x_2}, L'_2 = \frac{x_2 x_3}{x_1x_3} = \frac{x_2}{x_1}, L'_3 = \frac{\delta x_1x_3}{x_2^3}$, and for $i \geq 1,$ $\frac{L'_{2i+2}}{L'_{2i}} = \frac{x_2^2}{x_1^2} $ and $ \frac{L'_{2i+3}}{L'_{2i+1}} = \frac{x_1^2}{x_2^2}.$

Since these ratios are the same as in part 1, we can use the same reasoning to show that $[L_1',L_2',\ldots]$ converges for any choice of positive real numbers $x_1,x_2,x_3$.  
Therefore, we have that \begin{align*}
L' := \lim_{q \to \infty} \frac{x_{\frac{1}{q}}}{x_{\frac{1}{q-1}}} &= [L_1',L_2',L_3',\ldots]\\
&= L_1' + \frac{1}{[L_2',L_3',\ldots]} \\
\end{align*}

Then, we can compute $\overline{L}' = [L_2',L_3',\ldots]$ with Lemma  \ref{lem:TwoPeriodicCont} by setting $z_1 = L_2' = \frac{x_2}{x_1}$,  $z_2 = L_3' = \frac{\delta x_1x_3}{x_2^3}$ and $r = \frac{x_2^2}{x_1^2}$. 
By choosing the positive square root, we have that \[
\overline{L}' = \frac{x_2(x_2^2 + \delta x_3 - x_1^2) + x_2 \sqrt{(x_1^2 - x_2^2 - \delta x_3)^2 + 4 \delta x_1^2x_3}}{2\delta x_1x_3}.
\]
Since $L' = [L_1,\overline{L}']$, the statement follows after further algebraic manipulations. 
\end{proof}

We conclude the section by giving a linear recurrence which the sequence $\{m_{\frac{1}{q}}\}_q$ satisfies.  

\begin{prop}\label{prop:Recurrence1/q}
Set $m_{\frac{1}{0}} = 1$ and $m_{\frac{1}{1}} = 3$.  Then,  for all $q \geq 2$, 
\[
m_{\frac{1}{q}} = 5 m_{\frac{1}{q-1}} - m_{\frac{1}{q-2}} - 1
\]
\end{prop}

\begin{proof}
 In the Farey tree in Figure \ref{fig:MarkovAndQTrees},  we see that we reach a tuple with $\frac{1}{q}$ by replacing $\frac{1}{q-2}$ in the tuple $(\frac{0}{1},  \frac{1}{q-2},  \frac{1}{q-1})$.  Therefore, \[
m_{\frac{1}{q}} = \frac{m_{\frac{1}{q-1}}^2 + m_{\frac{1}{q-1}}m_{\frac{0}{1}} + m_{\frac{0}{1}}^2}{m_{\frac{1}{q-2}}}.
\]
Since $(m_{\frac{0}{1}},  m_{\frac{1}{q-2}},  m_{\frac{1}{q-1}})$ is a generalized Markov tuple,  we can use the generalized Markov equation to change the numerator, \begin{align*}
m_{\frac{1}{q}}& = \frac{6m_{\frac{1}{q-1}} m_{\frac{1}{q-2}} m_{\frac{0}{1}}   - m_{\frac{1}{q-1}} m_{\frac{1}{q-2}} - m_{\frac{1}{q-2}}^2  - m_{\frac{1}{q-2}} m_{\frac{0}{1}}}{m_{\frac{1}{q-2}}}\\
&= 5m_{\frac{1}{q-1}} - m_{\frac{1}{q-2}} - 1
\end{align*}

where we simply by canceling terms and using the fact that $m_{\frac{0}{1}} = 1$.  

\end{proof}

Note the sequence $\{n_{\frac{1}{q}}\}_q$  has the recurrence $n_{\frac{1}{q}} = 3 n_{\frac{1}{q-1}} - n_{\frac{1}{q-2}}$. 

\subsection{Generalized Markov numbers for $\frac{q-1}{q}$}\label{subsec:q-1q}

Here we present some analogues results to Section \ref{subsec:1q} for the sequence $\{m_{\frac{q-1}{q}}\}_q$.  

\begin{lemma}\label{lem:Patternq-1q}
Let $\beta$ be $[3,5]$.  Let $q \geq 2$. Then, if $q$ is even, $C_{\frac{q-1}{q}} = \mathcal{N}[\beta^{\frac{q-2}{2}},3, 4, (\beta^{-1})^{\frac{q-2}{2}}]$. If $q$ is odd, then $C_{\frac{q-1}{q}} = \mathcal{N}[\beta^{\frac{q-3}{2}},3, 5, 4,3,(\beta^{-1})^{\frac{q-3}{2}}]$.
\end{lemma}

\begin{proof}
First consider a portion of the arc $\gamma_{\frac{q-1}{q}}$ sufficiently to the left of the central crossing point.  Since the slope of the arc is closer to 1 than 0,  it passes close to the lattice points on its left and far from the lattice points on its right.  This produces the $3,5,3,5,\ldots$ behavior.  The central terms follow from Lemma \ref{lem:MidPointPattern}, picking a convention for the central crossing point, and by symmetry we see that to the right of the central crossing point we again have the pattern $3,5,3,5,\ldots$.  
\end{proof}

\begin{prop}\label{prop:Limitq-1q}
\[
\lim_{q \to \infty} \frac{m_{\frac{q-1}{q}}}{m_{\frac{q-2}{q-1}}} = \frac{17 + \sqrt{285}}{2}
\]
\end{prop}
\begin{proof}
Let $q \geq 3$.  First suppose that $q$ is odd.  Then following Lemma \ref{lem:Patternq-1q},  we have \[
[\beta^{\frac{q-3}{2}},3, 5, 4,3,(\beta^{-1})^{\frac{q-3}{2}}] = \frac{\mathcal{N}[\beta^{\frac{q-3}{2}},3, 5, 4,3,(\beta^{-1})^{\frac{q-3}{2}}]}{\mathcal{N}[5,\beta^{\frac{q-5}{2}},3, 5, 4,3,(\beta^{-1})^{\frac{q-3}{2}}]} = \frac{m_{\frac{q-1}{q}}}{\mathcal{N}[5,\beta^{\frac{q-5}{2}},3, 5, 4,3,(\beta^{-1})^{\frac{q-3}{2}}]} 
\]
and\[
[5,\beta^{\frac{q-5}{2}},3, 5, 4,3,(\beta^{-1})^{\frac{q-3}{2}}] = \frac{\mathcal{N}[5,\beta^{\frac{q-5}{2}},3, 5, 4,3,(\beta^{-1})^{\frac{q-3}{2}}]}{\mathcal{N}[\beta^{\frac{q-5}{2}},3, 5, 4,3,(\beta^{-1})^{\frac{q-3}{2}}]}.
\]
We maniuplate the denominator to reach a more familiar expression, \[
\mathcal{N}[\beta^{\frac{q-5}{2}},3, 5, 4,3,(\beta^{-1})^{\frac{q-3}{2}}] = [\beta^{\frac{q-3}{2}},4,3,(\beta^{-1})^{\frac{q-3}{2}}] \\=  [\beta^{\frac{q-3}{2}},3,4,(\beta^{-1})^{\frac{q-3}{2}}] = m_{\frac{q-1}{q}} \]
 where the last equality follows from Lemma \ref{lem:Patternq-1q}.  Therefore,  the desired ratio is a product of continued fractions \begin{align*}
\frac{m_{\frac{q-1}{q}}}{m_{\frac{q-2}{q-1}}} &= [\beta^{\frac{q-3}{2}},3, 5, 4,3,(\beta^{-1})^{\frac{q-3}{2}}] [5,\beta^{\frac{q-5}{2}},3, 5, 4,3,(\beta^{-1})^{\frac{q-3}{2}}] \\
&= \bigg(3 + \frac{1}{[5,\beta^{\frac{q-5}{2}},3, 5, 4,3,(\beta^{-1})^{\frac{q-3}{2}}]} \bigg) [5,\beta^{\frac{q-5}{2}},3, 5, 4,3,(\beta^{-1})^{\frac{q-3}{2}}]\\
&= 3 [5,\beta^{\frac{q-5}{2}},3, 5, 4,3,(\beta^{-1})^{\frac{q-3}{2}}] + 1
\end{align*}

We see that as $q$ goes to infinity,  the continued fraction will be 2-periodic and thus \[
\lim_{q \to \infty} \frac{m_{\frac{q-1}{q}}}{m_{\frac{q-2}{q-1}}} = 3[\overline{5,3}] + 1.
\]

From Lemma \ref{lem:TwoPeriodicCont},  we calculate $[\overline{5,3}] = \frac{15 + \sqrt{285}}{6}$,  and thus the statement follows.
\end{proof}

The corresponding limit in the ordinary Markov case is $\lim_{q \to \infty} n_{\frac{q-1}{q}}/n_{\frac{q-2}{q-1}} =  3 + 2\sqrt{2}$.


One could use the same ideas as in Proposition \ref{prop:InfClusterLimit} to compute the limit of generalized cluster variables $\lim_{q \to \infty} \frac{x_{\frac{q-1}{q}}}{x_{\frac{q-2}{q-1}}}$. As is necessary in the proof of \ref{prop:Limitq-1q},  this could be done by considering the product of two continued fractions of Laurent polynomials. We do not record this result as the polynomial has many terms and no clear factorization. 

We can, however, easily give a linear recurrence for the sequence $\{m_{\frac{q-1}{q}}\}_q$. 

\begin{prop}\label{prop:RecurrenceQ-1}
Set $m_{\frac{0}{1}} = 1$ and $m_{\frac{1}{2}} = 3$.  Then for $q \geq 3$,  we have
\[
m_{\frac{q-1}{q}} = 17 m_{\frac{q-2}{q-1}} - m_{\frac{q-3}{q-2}} - 3
\]
\end{prop}

\begin{proof}
Let $q \geq 3$.  We reach a Farey tuple with $\frac{q-1}{q}$ by exchanging $\frac{q-3}{q-2}$ in the tuple $(\frac{1}{1},  \frac{q-3}{q-2}, \frac{q-2}{q-1})$.  The proof follows the same reasoning as the proof of Proposition \ref{prop:Recurrence1/q},  using the fact that $m_{\frac{1}{1}} = 3$. 
\end{proof}

In the ordinary case, we have initial conditions $n_{\frac{0}{1}} = 1, n_{\frac{1}{2}} = 2$, and for $q \geq 3,$  the recurrence $n_{\frac{q-1}{q}} = 6n_{\frac{q-2}{q-1}} - n_{\frac{q-3}{q-2}}$. 

\subsection{Other interesting families}

The previous two sections looked at the sequences of generalized Markov numbers whose labels converged to 0 and 1. One could also pick a rational number $r$ between $0$ and $1$ and look at Markov numbers whose indices converge to $r$. However, when the number is strictly between 0 and 1, there will be two options; one sequence which approaches from above and one which approaches from below. For example, given the number $\frac12$, one could consider the sequences $\{m_{\frac{q}{2q-1}}\}_q$ and $\{m_{\frac{q}{2q+1}}\}_q$. By similar reasoning to Lemmas \ref{lem:Pattern1q} and \ref{lem:Patternq-1q}, for large $q$, $C_{\frac{q}{2q-1}}$ approaches the infinite continued fraction $[3, 4, 5, \overline{1, 2, 4, 5}]$ and $C_{\frac{q}{2q+1}}$ approaches  $[\overline{4,2,1,5}]$. One could use these to show growth behavior and, with some knowledge about how to form snake graphs from orbifolds, one could also compute limits of ratios of cluster variables corresponding to these infinite continued fractions. 

It is straightforward to give linear recurrences as well for these families. The terms approaching $\frac12$ from above can be found in Markov triples $(m_{\frac12}, m_{\frac{q-1}{2q-3}}, m_{\frac{q}{2q-1}})$. With the same reasoning as in Propositions \ref{prop:Recurrence1/q} and \ref{prop:RecurrenceQ-1}, we can show that \[
m_{\frac{q}{2q-1}} = 77 m_{\frac{q-1}{2q-3}} - m_{\frac{q-2}{2q-5}} - 13
\]
with initial conditions $m_{\frac12} = 13$ and $m_{\frac23} = 217$. The recurrence is in fact the same for the sequence $\{m_{\frac{q}{2q+1}}\}_q$, but the initial conditions are different. 

\section{Extending the Algorithm}\label{sec:extendalgorithm}

In order to provide a partial proof famous uniqueness conjecture, the authors of \cite{LLRS}  extend the correspondence between ordinary Markov numbers and rational numbers to include an assignment of numbers to integer points $(kq,kp)$ for $\gcd(p,q) = 1$ and $k \in \mathbb{Z}$.
Geometrically, they take the line segment between $(0,0)$ and $(kq,kp)$, and deform the line segment slightly to the left or right at each lattice point $(mq,mp)$ for $1 \leq m < k$. The deformation moves to the same side at each intermediate point. They show that the two snake graphs one gets from choosing a left or right deformation have the same number of perfect matchings, and they associate this number to the point $(kq,kp)$. When $k = 1$, this is just the Markov number $n_{\frac{q}{p}}$.

We consider these left or right deformed arcs in our setting as well. For clarity, we sometimes replace notation using $\frac{p}{q}$ with $(q,p)$ so that it is clear that our assignment of a number to $(kq,kp)$ is distinct from an assignment to $(q,p)$. Recall that if $\gamma_{\frac{p}{q}}$ is the line segment from $(0,0)$ to $(q,p)$ with $\gcd(p,q) = 1$, there is an  intersection between $\gamma_{\frac{p}{q}}$ and a line segment $\tau$ in the lattice which occurs at the midpoint  of $\tau$. This meant that the middle entry of $\mathbf{f}(\frac{p}{q})$ could be $+$ or $-$, and this choice would not affect the numerator of the continued fraction associated to $\mathbf{f}(\frac{p}{q})$. However, now if we consider an arc from $(0,0)$ to $(kq,kp)$ with $k > 1$, and if we deform our arc to the left (right) at the points $(iq,ip)$ for $1 \leq i \leq k-1$, we will always assign $+$ ($-$) to the central crossing point between each pair of lattice points $((i-1)q,(i-1)p)$ and $(iq,ip)$ since the deformation means this crossing point is now slightly to the left (right) of the midpoint. Then, by considering the crossing sequence as our arc makes a small half circle to the left of a lattice point, we see that if $C_{(q,p)} = C_{\frac{p}{q}} = [a_1,\ldots,a_n]$, then $C^L_{(kq,kp)} = [a_1,\ldots,a_n,5,1,a_1-1,\ldots,a_n,5,\ldots,1,a_1-1,\ldots,a_n]$ such that there are  $k-1$ entries of $5$ that connect two segments of the form $a_1,\ldots,a_n$ or $1,a_1 - 1, \ldots,a_n$.  If we instead consider deforming the midpoint to the right, then $C^R_{(kq,kp)} = [a_1,\ldots,a_n - 1,1,5,a_1,\ldots,a_n-1,1,5,a_1,\ldots,a_n]$. 

The sequence $C^L_{(kq,kp)}$ is equal to the reversal of $C^R_{(kq,kp)}$.  For example, consider $\frac{p}{q} = \frac23$. Then, since $C_{\frac{2}{3}} = [3,4,5,3]$ or $[3,5,4,3]$, we have that $C^L_{(kq,kp)} = [3,5,4,3,5,1,2,5,4,3]$ and $C^R_{(kq,kp)} = [3,4,5,2,1,5,3,4,5,3]$. By Lemma \ref{lem:FlipContFraction}, the numerators of these continued fractions are the same; hence, the choice of deforming to the right or left will not affect the assignment of the number $m_{(kq,kp)}$. 
We choose to use left deformations as a convention. 

For positive integers $p,q,k$ with $\gcd(p,q) = 1$, we set $m_{(kq,kp)} = \mathcal{N}(C^L_{(kq,kp)})$. In the case of ordinary Markov numbers and a cluster algebra from a torus, the left or right deformed arc between $(0,0)$ and $(kq,kp)$ for $k > 1$ would correspond to an arc on the torus with $(k-1)$ self-intersections. Thus, we could apply skein relations to resolve the self-intersections; this provides relations amongst the numbers $n_{(kq,kp)}$. We show that our orbifold Markov numbers and their extensions, $m_{(kq,kp)}$, satisfy the same types of relations. Since the resolution of a self-intersection results in a closed curve, we first discuss band graphs. 

\subsection{Generalized Markov Band Graphs}

In this section, we define band graphs and give a formula for computing the number of perfect matchings of such graphs. For notation, let $G = (G_1,\ldots,G_m)$ be a snake graph on $m$ tiles. Let $S(G_i)$ be the south edge of $G_i$, and define $W(G_i),N(G_i),E(G_i)$ similarly. 

We also recall the notion of minimal and maximal matchings. In \cite{MSW2}, the authors show that the set of perfect matchings of a snake graph form a lattice. The minimal and maximal elements of this lattice correspond to the two matchings of a snake graph which only use \emph{boundary} edges; these are edges which only border one tile of the graph. Deciding which is the minimal matching is up to a convention; we will use the convention that $S(G_1)$ is always an edge in the minimal matching. 

Band graphs were introduced in \cite{MSW2} to describe bases of cluster algebras of surface type.
Given a snake graph $G = (G_1,\ldots,G_m)$, we form a \emph{band graph} by gluing one of the edges $e \in \{S(G_1),W(G_1)\}$ with one of the edges $e' \in \{N(G_m),E(G_m)\}$ in such a way that $e$ is in the minimal matching if and only if $e'$ is not in the minimal matching. Therefore, there are two ways to form a band graph from a given snake graph $G$. 

Given a band graph $G$, with vertices $x,y$ along the glued edge $e = e'$, we say a perfect matching $P$ is a \emph{good matching} if $e \in P$ or if the two edges in $P$ which are adjacent to $x$ and $y$, lie on the same side of the glued edge $e = e'$. Good matchings of a band graph $\widetilde{G}$ correspond to perfect matchings of the underlying snake graph $G$ which use $e$ or $e'$. By construction, the minimal matching of $G$ uses $e$ and the maximal matching uses $e'$, or vice versa, so these always descend to good matchings of the band graph $\widetilde{G}$.

 For convenience we introduce the idea of a \emph{dominant} edge.

\begin{definition}
Let $e \in \{S(G_1),W(G_1)\}$ be the unique edge whose two vertices are both only adjacent to $G_1$. We call $e$ \emph{dominant}, and we call the other edge in $\{S(G_1),W(G_1)\}$ \emph{non dominant}. We similarly call the unique edge in $\{N(G_m),E(G_m)\}$ with both vertices only adjacent to $G_m$ \emph{dominant}.
\end{definition}

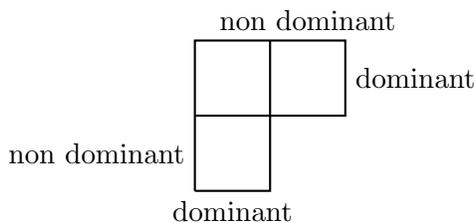
\begin{center}
\begin{tikzpicture}
\draw[thick] (0,0) to (1,0) to (1,1) to (2,1) to (2,2) to (1,2) to (0,2) to (0,1) to (0,0);
\draw[thick] (0,1) to (1,1) to (1,2);
\node[below] at (0.5,0){dominant};
\node[left] at (0,0.5){non dominant};
\node[right] at (2,1.5){dominant};
\node[above] at (1.5,2){non dominant};
\end{tikzpicture}
\end{center}

Each snake graph has two dominant edges. The following is straightforward.

\begin{lemma}\label{lem:DominantEdges}
Let $G = \mathcal{G}[a_1,\ldots,a_n]$ where $a_1 > 1$ and $a_n > 1$.
\begin{itemize}
    \item If $n$ is even, then either the minimal matching uses both dominant edges and the maximal matching uses neither dominant edge or vice versa.
    \item If $n$ is odd, then the minimal and maximal matchings each use exactly one dominant edge. 
\end{itemize}
\end{lemma}

\begin{proof}
We induct on $n$; the values of $a_i$ will not affect the statement. The claim is immediately true for $n=1$ by analyzing minimal and maximal matchings on a zig-zag snake graph. 

Now assume we have shown the claim for snake graphs $\mathcal{G}[a_1,\ldots,a_{n-1}]$ for any choices of $a_i \in \mathbb{Z}_{>0}$, and consider a snake graph  $\mathcal{G}[a_1,\ldots,a_{n}]$. Assume that $n$ is even. Then, we know that the minimal matching of $G' = \mathcal{G}[a_1,\ldots,a_{n-1}+1]$ uses exactly one dominant edge. Suppose that this is the dominant edge on $G_1$, so that we use the non dominant edge on the final tile of $G'$. Call this final tile $G_{m'}$. We form $G$ by gluing the non-dominant edge of the first tile of $G'' = \mathcal{G}[a_n]$ onto the dominant edge on $G_{m'}$. Thus, we can complete the minimal matching on $G$ by taking the minimal/maximal matching on $G''$ which uses the non dominant edge on the first tile. By the base case, this matching of $G''$ uses the dominant edge on its last tile. Therefore, in this case the minimal matching on $G$ uses both dominant edges, which immediately implies that the maximal matching uses neither. The other cases can be proven similarly. 
\end{proof}

By considering the fact that the minimal and maximal matchings on $G$ descend to good matchings for either choice of band graph coming from $G$, we have the following corollary to Lemma \ref{lem:DominantEdges}.

\begin{cor}\label{cor:DominantEdgesGluing}
Let $G = \mathcal{G}[a_1,\ldots,a_n]$. 
\begin{itemize}
    \item If $n$ is even, then either band graph arising from $G$ involves gluing exactly one dominant edge.
    \item If $n$ is odd, then one band graph arising from $G$ involves gluing neither dominant edge and the other involves gluing both dominant edges. 
\end{itemize}
\end{cor}

By Corollary \ref{cor:DominantEdgesGluing}, a band graph formed from $\mathcal{G}[a_1,\ldots, a_n]$ is determined by knowing whether or not the dominant edge on $G_1$ is the glued edge. We let $\mathcal{G}^\circ_D[a_1,\ldots,a_n]$ denote the band graph which involves gluing the dominant edge on $G_1$ and $\mathcal{G}^\circ_N[a_1,\ldots,a_n]$ denote the band graph which involves gluing the non-dominant edge on $G_1$. Similarly, let $\mathcal{N}^\circ_X[a_1,\ldots,a_n]$ be the number of good matchings of $\mathcal{G}^\circ_X[a_1,\ldots,a_n]$ for $X$ either $D$ or $N$. 

\begin{prop}\label{prop:BandMatchings}
Let $a_1,\ldots,a_n$ be positive integers with $a_1 > 1$ and $a_n > 1$.
\begin{enumerate}
    \item If $n$ is even, then \[
    \mathcal{N}^\circ_D[a_1,\ldots,a_n] = \mathcal{N}[a_1,\ldots,a_n] - \mathcal{N}[a_2,\ldots,a_n-1]
    \]
    and 
    \[
    \mathcal{N}^\circ_N[a_1,\ldots,a_n] =\mathcal{N}[a_1,\ldots,a_n] - \mathcal{N}[a_1-1,\ldots,a_{n-1}]
    \]
    \item If $n$ is odd , then \[
    \mathcal{N}^\circ_D[a_1,\ldots,a_n] =\mathcal{N}[a_1,\ldots,a_n] - \mathcal{N}[a_2,\ldots,a_{n-1}]
    \]
    and \[
    \mathcal{N}^\circ_N[a_1,\ldots,a_n] =\mathcal{N}[a_1,\ldots,a_n] - \mathcal{N}[a_1-1,\ldots,a_n-1]
    \]
    Note that in the $n = 1$ case, $\mathcal{N}^\circ_D[a_1] = a_1$ and $\mathcal{N}^\circ_N[a_1] = 2$.  
\end{enumerate}
\end{prop}

\begin{proof}
Each case is proven by counting the number of matchings of the un-glued snake graph $G$ which do not lift to a good matching of the band graph $B$. These are exactly the matchings which do not use either of the edges which are identified to form the band graph. When we have a dominant edge in the gluing, then not using this dominant edge forces a minimal/maximal matching on the first or last section of $G$; hence, we are ignoring the first or last entry of the continued fraction. When a nondominant edge is in the gluing, then if we instead use the adjacent dominant edge we have not forced any additional edges to be used in the matching. 
\end{proof}

Other formulas for computing the number of good matchings of a band graph will appear in \cite{Apruzzese}.

We use the results above to analyze the number of perfect matchings of band graphs coming from the arcs $\gamma_{(q,p)} = \gamma_{\frac{p}{q}}$. Following Section 4.3 of \cite{CS}, we take the arc $\gamma_{\frac{p}{q}}$ and nudge both endpoints an infinitesimal amount away from the lattice point to the points $(\epsilon,\epsilon)$ and $(q+\epsilon,p+\epsilon)$ for small $\epsilon>0$; we also change the arc so that it no longer passes through $(q,p)$ by nudging the path to the left or right, as in the case of arcs $\gamma_{(kq,kp)}$. Then, we identify these points; this corresponds to a simple (i.e. without self-intersections) closed curve in the orbifold $\mathcal{O}_3$.

Let $\gamma^{L, \circ}_{\frac{p}{q}}$ and $\gamma^{R, \circ}_{\frac{p}{q}}$ be these two arcs with identified endpoints, and let $m^{L, \circ}_{\frac{p}{q}}$ and $m^{R, \circ}_{\frac{p}{q}}$ be the number of good matchings of the corresponding band graphs. Call these band graphs $G^{L,\circ}_{\frac{p}{q}}$ and $G^{R,\circ}_{\frac{p}{q}}$. The proof of Theorem \ref{thm:BandMarkov} will describe the structure of these graphs and show that $m^{L,\circ}_{\frac{p}{q}} = m^{R,\circ}_{\frac{p}{q}}$.

\begin{theorem}\label{thm:BandMarkov}
For $p,q \in \mathbb{Z}_{>0}$ with $p \leq q$ and $\gcd(p,q) = 1$, we have \[
m^{L, \circ}_{\fpq} = m^{R, \circ}_{\fpq}= 6 m_\fpq -1.
\]
\end{theorem}

\begin{proof}
We consider $p = q=1$ separately.  The arc $\gamma_{\frac{1}{1}}$ only crosses one arc, $\tau_{\frac{-1}{1}}$, so $G_{\frac{1}{1}} = \mathcal{G}[3]$. The sign sequence for $\gamma^{L,\circ}_{\frac{1}{1}}$, keeping the convention $a_1 > 1$ and $a_n>1$, is $+++------$. Since the first arc that $\gamma^{L,\circ}_{\frac{1}{1}}$ crosses is $\tau_{\frac{-1}{1}}$ and the last arc is $\tau_{\frac{1}{0}}$, we will form the graph $G^{L,\circ}_{\frac{1}{1}}$ by gluing the end tiles on the edges labeled $\tau_{\frac{0}{1}}$. Thus, we can compute $m_{\frac{1}{1}}^{L,\circ}$ once we determine whether this edge is dominant on the first tile. Using the convention that the orientation of the labels of the first tile should match the orientation of the lattice (and of $\mathcal{O}$), we see that the edge labeled $x_{\frac{0}{1}}$ on the first tile is the non-dominant edge. By Proposition \ref{prop:BandMatchings}, $m^{L,\circ}_{\frac{1}{1}} = \mathcal{N}^\circ_N[3,6] = \mathcal{N}[3,6] - \mathcal{N}[2] = 17.$

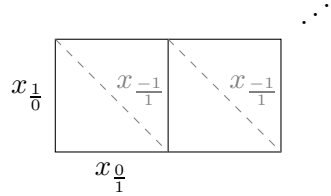
\begin{center}
\begin{tikzpicture}[scale = 1.5]
\draw (6,1) to(7,1) to  (7,2) to   (6,2) to  node[above]{} (5,2) to node[left]{$x_{\frac{1}{0}}$} (5,1) to node[below]{$x_{\frac{0}{1}}$}  (6,1);
\draw (6,1) --(6,2) ;
\draw[dashed,gray] (5,2) to node[right,  yshift = 2pt, xshift = -2pt ]{$x_{\frac{-1}{1}}$} (6,1);
\draw[dashed,gray] (6,2) to node[right,  yshift = 2pt, xshift = -2pt]{$x_{\frac{-1}{1}}$} (7,1);
\node[] at (7.3,2.3){$\iddots$};
\end{tikzpicture}
\end{center}

If we instead consider $\gamma^{R,\circ}_{\frac{1}{1}}$, then the sign sequence is $---++++++$. The last arc that $\gamma^{R,\circ}_{\frac{1}{1}}$ crosses is instead $\tau_{\frac{0}{1}}$, and the first two tiles of $G^{R,\circ}_{\frac{1}{1}}$ are glued vertically instead of horizontally. Thus, $m^{L,\circ}_{\frac{1}{1}} = \mathcal{N}^\circ_N[3,6]$ as well.

Now consider $\frac{p}{q} < 1$. The first arc $\gamma^{L,\circ}_\fpq$ crosses is $\tau_{\frac{-1}{1}}$ and the last arc $\gamma^{L,\circ}_\fpq$ crosses is $\tau_{\frac{1}{0}}$. Thus, we know that the band graph $G^\circ_\fpq$ glues along $\tau_{\frac{0}{1}}$. For $\frac{p}{q}<1$ the first two entries of the sign sequence will be $--$, regardless of the direction of nudging. This implies that the first two tiles will always be glued vertically, as below. Thus, we will be gluing along the dominant edge on the first tile.

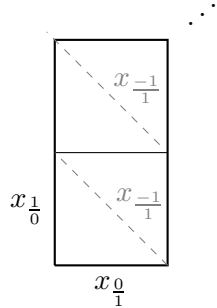
\begin{center}
\begin{tikzpicture}[scale = 1.5]
\draw[thick] (0,0) to node[below]{$x_{\frac{0}{1}}$} (1,0) -- (1,1) -- (1,2) -- (0,2) -- (0,1) to node[left]{$x_{\frac{1}{0}}$} (0,0);
\draw (0,1) -- (1,1);
\draw[gray, dashed] (1,0) to node[right, yshift = 2pt, xshift = -2pt]{$x_{\frac{-1}{1}}$} (0,1);
\draw[gray, dashed,  yshift = 2pt, xshift = -2pt] (1,1) to node[right]{$x_{\frac{-1}{1}}$} (0,2);
\node[] at (1.3,2.3){$\iddots$};
\end{tikzpicture}
\end{center}

We next evaluate how the continued fraction associated to $\gamma_{\fpq}^{L,\circ}$ compares with that for $\gamma_{\fpq}$. As mentioned above, at the central crossing point for $\gamma_{\fpq}$, we now use $+$ since this crossing is now slightly closer to the left endpoint. Then, by similar reasons as for arcs $\gamma^L_{(kq,kp)}$, we see that if $C_{\frac{p}{q}} = [a_1,\ldots,a_1]$, then $C_{\frac{p}{q}}^{L,\circ} = [a_1,\ldots,a_1,6]$ where $C_{\frac{p}{q}}^{L,\circ}$ gives the continued fraction from the sign sequence for the arc $\gamma^L_{(kq,kp)}$ ignoring the identification.  Thus, we have that $m_{\fpq}^{L,\circ} = \mathcal{N}_D^\circ[a_1,\ldots,a_1,6]$. 

We know that $n$ is even, so we can compute this using Proposition \ref{prop:BandMatchings},
\[
\mathcal{N}^\circ_D[a_1,\ldots,a_1,6] = \mathcal{N}[a_1,\ldots,a_1,6] - \mathcal{N}[a_2,\ldots, a_1].
\]
By Lemma \ref{lem:ContFracRecurrence}, we further manipulate the first term, 
\[
\mathcal{N}^\circ_D[a_1,\ldots,a_1,6] = 6 \mathcal{N}[a_1,\ldots,a_1] + \mathcal{N}[a_1,\ldots,a_2] - \mathcal{N}[a_2,\ldots,a_1].
\]
Since $\mathcal{N}[a_1,\ldots,a_1] = m_{\frac{p}{q}}$, we are left with showing $\mathcal{N}[a_1,\ldots,a_2] - \mathcal{N}[a_2,\ldots,a_1] = -1$. First, suppose that the number of terms in $a_1,\ldots,a_2$ is $4\ell -1$ for some $\ell \geq 1$. Then, we have that $C_{\frac{p}{q}}^{L,\circ} = a_1,a_2,\ldots, a_{2\ell} + 1, a_{2\ell}, \ldots, a_2,a_1,6$, since even-indexed entries count the length of subsequences of $+$'s and  we assign $+$ to the central crossing point. Thus, by Lemma \ref{lem:SwapPlus1}, $\mathcal{N}[a_1,\ldots,a_2] - \mathcal{N}[a_2,\ldots,a_1] = -c_{2\ell} = -1$.

Next, suppose that the number of terms in $a_1,\ldots,a_2$ is $4\ell + 1$. Then, $C_{\frac{p}{q}}^{L,\circ} = a_1,a_2,\ldots, a_{2\ell+1}, a_{2\ell+1}+1, \ldots, a_2,a_1,6$, so \begin{align*}
\mathcal{N}[a_1,\ldots,a_{2\ell+1}, a_{2\ell+1}+1, \ldots, a_2]  - \mathcal{N}[a_2,\ldots,a_{2\ell+1}, a_{2\ell+1}+1, \ldots, a_1]\\
=\mathcal{N}[a_2,\ldots,a_{2\ell + 1} + 1, a_{2\ell + 1},\ldots,a_1] - \mathcal{N}[a_1,\ldots,a_{2\ell + 1} + 1, a_{2\ell + 1},\ldots,a_1]\\ = c_{2\ell + 1} = -1,
\end{align*}

where we again apply Lemma \ref{lem:SwapPlus1}. 

Next we turn to $\gamma_{\frac{p}{q}}^{R,\circ}$. As when analyzing arcs $\gamma_{(kq,kp)}^R$, we have that if $C_{\frac{p}{q}} = [a_1,\ldots,a_1]$, then $C_{\frac{p}{q}}^{R,\circ} = [a_1,\ldots,a_1-1,1,6]$. In this case, the first arc $\gamma_{\fpq}^{R,\circ}$ crosses is $\tau_{\frac{-1}{1}}$ and the last arc crossed is $\tau_{\frac{0}{1}}$. Therefore, we form $G_{\fpq}^{R,\circ}$ by gluing along the edge labeled with $\tau_{\frac{1}{0}}$ on the first tile; this is the non-dominant edge. Therefore, $m_{\fpq}^{R,\circ} = \mathcal{N}_N^\circ[a_1,\ldots,a_1-1,1,6]$. Since the number of terms is even, by Proposition \ref{prop:BandMatchings} we have that \begin{align*}
\mathcal{N}_N^\circ[a_1,\ldots,a_1-1,1,6] &= \mathcal{N}[a_1,\ldots,a_1-1,1,6] - \mathcal{N}[a_1-1,\ldots,a_1-1,1]\\
&= 6\mathcal{N}[a_1,\ldots,a_1-1,1] + \mathcal{N}[a_1,\ldots,a_1-1] \\&- \mathcal{N}[a_1-1,\ldots,a_1-1,1]\\
&= 6 m_{\fpq} + \mathcal{N}[a_1,\ldots,a_1-1] - \mathcal{N}[a_1-1,\ldots,a_1]
\end{align*}

We can rewrite \[
\mathcal{N}[a_1,\ldots,a_1-1] - \mathcal{N}[a_1-1,\ldots,a_1] = \mathcal{N}[1,a_1-1,\ldots,a_1-1] - \mathcal{N}[a_1-1,\ldots,a_1-1,1],
\]

and then by a similar argument as before, using Lemma \ref{lem:SwapPlus1}, we can show that $\mathcal{N}[1,a_1-1,\ldots,a_1-1] - \mathcal{N}[a_1-1,\ldots,a_1-1,1] = -1$.
\end{proof}

Seeing in Theorem \ref{thm:BandMarkov} that our choice of direction to navigate around the lattice point does not matter, we now will write $m_{\fpq}^{\circ} = m_{\fpq}^{L,\circ} = m_{\fpq}^{R,\circ}$. In \cite{CS}, the authors show that for ordinary Markov numbers and arcs on a torus, $n^\circ_\fpq = 3n_\fpq$. 


\subsection{Recurrence on $m_{(kq,kp)}$}

In the following, we show how $m_{(kq,kp)}$ compares with $m_{(q,p)}$. 

\begin{theorem}\label{thm:Recurrence}
Let $\frac{p}{q} \leq 1$ have $\gcd(p,q) = 1$. Let $k \geq 2$, and define $m_0 = 0$. Then, 
\[
m_{\frac{kp}{kq}} = m^\circ_{\frac{p}{q}} m_{\frac{(k-1)p}{(k-1)q}} - m_{\frac{(k-2)p}{(k-2)q}}.
\]
\end{theorem}

\begin{proof}

We will show this relation by using \emph{snake graph calculus}, which was introduced in \cite{CS4}; we will look at resolving snake graphs with ``self-intersection'', which was explained in a follow-up work \cite{CS2}. We will use notation from \cite{CS2} and invite the interested reader to consult this source for more precise definitions. 

Given a snake graph $G = (G_1,\ldots,G_d)$, for $i \leq j$, let $G[i,j]$ be the subgraph $(G_i,\ldots,G_j)$.

We consider the $k = 2$ case separately. In this case, we expect that $m_{(2q,2p)} = m_{(q,p)}^\circ m_{(q,p)}$. If $G_{(q,p)}$ has $d$ tiles, then $G_{(2q,2p)}$ has $2d + 6$ tiles. Naturally, the subgraphs $G_{(2q,2p)}[1,d]$ and $G_{(2q,2p)}[d+7,2d+6]$ are isomorphic. Call this graph $\mathcal{G}$. The two inclusions $i_1$ and $i_2$ a of $\mathcal{G}$ into $G_{(kq,kp)}$ are maximal in the sense that there is no pair of larger isomorphic subgraphs of $G_{(2q,2p)}$ which contain $i_1(\mathcal{G})$ and $i_2(\mathcal{G})$.  This shows $G_{(2q,2p)}$ has a \emph{self-overlap} in the sense of \cite{CS2}. Since the inclusions $i_1,i_2$ both map the southwest-most tile of $\mathcal{G}$ to the southwest most tile of the corresponding subgraphs, this self-overlap is \emph{in the same direction}. 
We denote by $s$ and $t$ the labels of the first and last tiles of $i_1(\mathcal{G})$; that is,  $s = 1$ and $t = d$. Similarly, let $s' = d+7$ and $t' = 2d+6$ be the same indices for the first and last tiles of $i_2(\mathcal{G})$.

Next, we check that $G_{(2q,2p)}$ satisfies Definition 2.6 in \cite{CS2}. Since $s = 1$ and $t' = 2d+6$, where our graph has $2d+6$ tiles, we check that the sign on the internal edge between $G_{t}$ and $G_{t+1}$ (call this edge $e_t$) is the same as the sign on the internal edge between $G_{s'-1}$ and $G_s$ ($e_{s'-1}$).These signs correspond to the entries $f_{t}$ and $f_{s'-1}$ in the sign function $\mathbf{f}((2q,2p))$. By our convention, both of these entries will have sign $+$, since they correspond to a crossings right before and after the small circle $\gamma_{(2q,2p)}$ takes around $(q,p))$.

Thus, $G_{(2q,2p)}$ \emph{self-crosses}. This allows us to follow the construction in \cite{CS2} of a \emph{resolution} of the crossing; this will give an equation relating the number of perfect matchings of several snake graphs. 

\begin{itemize}
    \item First, we have that  $\mathcal{G}_3 = G[s,t] \cup G[t'+1,2d+6] = G[1,d]$; this is isomorphic to $G_{\frac{p}{q}}$. By definition, the number of matchings of this subgraph is  $m_{\frac{p}{q}}$. 
    \item Next, $\mathcal{G}_4^\circ $ is the band graph with underlying snake graph $G[s, s'-1] = G[1,d+6]$ and (using a left nudge) which glues on the edges labeled with $\tau_{\frac{0}{1}}$. From Theorem \ref{thm:BandMarkov},  the number of good matchings of this snake graph is $m^\circ_{(q,p)}$. 
    \item Since $s = 1$ and $t = 2d+6$, where our graph has $2d+6$ tiles, to find $\mathcal{G}_{56}$ we are in subcase 2d of case 1 in Section 3 of \cite{CS2}. Then,  $\mathcal{G}_{56}' = \overline{G}[d+7,d+1]$ where the overline denotes a reversal of the snake graph; this is a zig-zag on 6 tiles.  We have that $\mathcal{G}_{56}$ is the section of $\mathcal{G}_{56}'$ between the first edge with the same sign as $e_{s'-1}$ and the last edge with the same sign as $e_t$.  As discussed, each of these signs is $+$. However, every interior edge of $\mathcal{G}_{56}'$ has sign $-$ since this subgraph corresponds to the portion of the arc which forms a small half-circle to the left of a point on the lattice, hence closer to the right endpoints of the arcs crossed. Therefore, $\mathcal{G}_{56}$ is an empty graph with zero perfect matchings.
\end{itemize}

By Theorem 4.5 in \cite{CS2}, we conclude that $m_{(2q,2p)} = m_{\frac{p}{q}} m_{\frac{p}{q}}^\circ$. 

 Now we turn to the case for $k > 2$.  The snake graph $G_{(kq,kp)}$ has $(6+d)k -6$ tiles where each subgraph $G_{(q,p)}[(d+6)(i-1) + 1, (d+6)i - 6]$ for $1 \leq i \leq k$ is isomorphic to $G_{(q,p)}$. The connected subgraphs of the complement, of the form $G[(d+6)i - 5, (6+d)i]$  for $1 \leq i \leq k-1$, each form a zig-zag shape. Therefore, the subgraphs $G[1,(d+6)(k-1) - 6]$ and $G[d+7, \ldots, (6+d)k -6]$ are isomorphic. Call this subgraph $\mathcal{G}$; by the same reasoning as in the $k=2$ case, we see that $G_{(kq,kp)}$ has a self-overlap in the same direction, and this counts as a self-crossing. Here, we set $s = 1, t =(d+6)(k-1) - 6, s' = d+7,$ and $t' = (d+6)k -6$. The computations for $\mathcal{G}_3$ and $\mathcal{G}_4^\circ$ are similar to the $k=2$ case while the analysis for the graph $\mathcal{G}_{56}$ is different.
 
 \begin{itemize}
     \item The graph $\mathcal{G}_3$ is $G[1,t] = [1,(d+6)(k-1)-6]$, which is the same as the graph $G_{((k-1)q,(k-1)p)}$. Thus, $m(\mathcal{G}_3) = m_{((k-1)q,(k-1)p)}$.
     \item The band graph $\mathcal{G}_4^\circ$ has underlying graph $G[s,s'-1] = G[1,d+6]$ and is glued on the edges on the extreme tiles labeled with $\tau_{\frac{0}{1}}$. Just as in the $k = 2$ case, $m(\mathcal{G}_4^\circ) = m_{(q,p)}^\circ$.
     \item Since $s' < t$ for $k > 2$, we are now in Subcase 1 of Case 1 in Section 3.2 of \cite{CS2}, so $\mathcal{G}_{56} = -G[s',t] = -G[d+7,(d+6)(k-1)-6]$, where we have that $m(-G) = -m(G)$. Thus, we have that $m(\mathcal{G}_{56}) = -m_{(k-2)q,(k-2)p}$.  
 \end{itemize}
 
 Therefore, by Theorem 4.12 in \cite{CS2}, we have that \[m_{(kq,kp)} = m(G_{(kq,kp)}) = m_{((k-1)q,(k-1)p)}m_{(q,p)}^\circ -  m_{((k-2)q,(k-2)p)}.
 \]

\end{proof}

Define the family of (normalized) Chebyshev polynomials of the second kind by $U_0(x) = 1, U_1(x) = x$, and for $k \geq 2$, $U_k(x) = xU_{k-1}(x) - U_{k-2}(x)$. By repeated use of Theorem \ref{thm:Recurrence} and Theorem \ref{thm:BandMarkov} we can also express $m_{(kq,kp)}$ completely in terms of $m_{(q,p)} = m_{\fpq}$.

\begin{cor}
Let $p \leq q$ satisfy $\gcd(p,q) = 1$. Let $k \geq 1$. Then, \[
m_{(kq,kp)} = U_{k-1}(m^\circ_{(q,p)}) m_{(q,p)} = U_{k-1}(6m_{(q,p)} - 1) m_{(q,p)} .
\]
\end{cor}

For example,  from Table \ref{table:GenMarkov},  $m_{3,3} = \mathcal{N}[3,5,3,53] = 846 = U_2(m^\circ_{(1,1)}) m_{(1,1)} =3 U_2(17)$.

\section*{Acknowledgements}

The first author would like to thank Gregg Musiker for the original idea to work on these generalizations of Markov numbers and Yasuaki Gyoda for discussions about the project.

  \end{document}